\documentclass{amsart}
\makeatletter
\@namedef{subjclassname@2020}{%
  \textup{2020} Mathematics Subject Classification}
\makeatother % Cambiar el año de la MSC, de 2010 a 2020

\usepackage{amsthm,amssymb,amsfonts,latexsym,mathtools,thmtools}
\usepackage[T1]{fontenc}
\usepackage{mathrsfs}
\usepackage{tikz-cd} % Commutative diagrams.
\usepackage{enumitem} % Customization of items.
\usepackage{hyperref} %hyperlinks
\usepackage[all]{xy}
\hypersetup{
    colorlinks=true,
    linkcolor=blue,
    filecolor=blue,      
    urlcolor=cyan,
    linktocpage=true
}

\newtheorem{theorem}{Theorem}[section]
\newtheorem{lemma}[theorem]{Lemma}

\theoremstyle{definition}
\newtheorem{definition}[theorem]{Definition}
\newtheorem{example}[theorem]{Example}

\theoremstyle{remark}
\newtheorem{remark}[theorem]{Remark}

\newtheorem{proposition}[theorem]{Proposition}
\numberwithin{equation}{section}

\begin{document}

\title[Serre-Artin-Zhang-Verevkin theorem for semi-graded rings]{Noncommutative scheme theory and the Serre-Artin-Zhang-Verevkin theorem for semi-graded rings}

%    Remove any unused author tags.

%    author one information

\author{Andr\'es Chac\'on}
\address{Universidad Nacional de Colombia - Sede Bogot\'a}
\curraddr{Campus Universitario}
\email{anchaconca@unal.edu.co}
\thanks{}

%   author two information
\author{Armando Reyes}
\address{Universidad Nacional de Colombia - Sede Bogot\'a}
\curraddr{Campus Universitario}
\email{mareyesv@unal.edu.co}

\thanks{The authors were supported by the research fund of Faculty of Science, Code HERMES 53880, Universidad Nacional de Colombia - Sede Bogot\'a, Colombia.}



\date{}

\dedicatory{Dedicated to Professor Lorenzo Acosta} 

\begin{abstract}

In this paper, we present a noncommutative scheme theory for the semi-graded rings generated in degree one defined by Lezama and Latorre \cite{LezamaLatorre2017} following the ideas about schematicness introduced by Van Oystaeyen and Willaert \cite{VanOystaeyenWillaert1995} for $\mathbb{N}$-graded algebras. With this theory, we prove the Serre-Artin-Zhang-Verevkin theorem \cite{ArtinZhang1994, EGAII1961, Hartshorne1977, Serre1955, Verevkin1992a, Verevkin1992} for several families of non-$\mathbb{N}$-graded algebras and finitely non-$\mathbb{N}$-graded algebras appearing in ring theory and noncommutative algebraic geometry. Our treatment contributes to the research on this theorem presented by Lezama \cite{Lezama2021, LezamaLatorre2017} from a different point of view.

\end{abstract}

\maketitle

\section{Introduction}

In his beautiful paper \cite{Serre1955}, Serre proved a theorem that describes the coherent sheaves on a projective scheme in terms of graded modules. Briefly, for $A$ a finitely generated commutative graded $\Bbbk$-algebra ($\Bbbk$ a field) and $X$ the associated projective scheme, if ${\rm coh}\ X$ denotes the category of coherent sheaves on $X$ and $\mathcal{O}_X(n)$ is the $n$th power of the twisting sheaf on $X$ \cite[p. 117]{Hartshorne1977}, we have a functor $\Gamma_*: {\rm coh}\ X \to {\rm qgr}\ A$ given by
\[
\Gamma_*(\mathcal{F}) = \bigoplus_{d = -\infty}^{\infty} {\rm H}^0(X, \mathcal{F}\otimes \mathcal{O}_X(d)).
\]

{\em Serre's theorem} \cite[Section 59, Proposition. 7.8, p. 252]{Serre1955}, \cite[3.3.5]{EGAII1961} and \cite[Proposition. II. 5.15]{Hartshorne1977}, asserts that if $A$ is generated over $\Bbbk$ by elements of degree one, then $\Gamma_*$ defines an equivalence of categories ${\rm coh}\ X\to {\rm qgr}\ A$.

\medskip

Artin and Zhang \cite{ArtinZhang1994} extended Serre's theorem to the noncommutative setting in the following way: let $A$ be an $\mathbb{N}$-graded algebra over a commutative Noetherian ring. They defined the associated projective scheme to be the pair ${\rm Proj}\ A = ({\rm qgr}\ A, \mathcal{A})$, where ${\rm qgr}\ A$ is the quotient category above and $\mathcal{A}$ is the object determined by the right module $A_A$. They showed the autoequivalence $s$ of ${\rm qgr}\ A$ defined by the shift of degrees in ${\rm gr}\ A$. The object $\mathcal{A}$ plays the role of the {\em structure sheaf} of ${\rm Proj}\ A$ and $s$ the role of the {\em polarization} defined by the projective embedding (this definition is the same as is given by Verevkin \cite{Verevkin1992a, Verevkin1992}). Since Serre's theorem does not hold for all commutative graded algebras, i.e., the functor defined by $\Gamma_*$ need not be an equivalence, Artin and Zhang's definition of ${\rm Proj}\ A$ is compatible with the classical definition for commutative graded rings only under some additional hypotheses, such as that $A$ is generated in degree one. In the literature, the noncommutative version of Serre's theorem is known as {\em Serre-Artin-Zhang-Verevkin theorem} \cite{ArtinZhang1994, Verevkin1992a, Verevkin1992}. Several authors have investigated results of commutative algebraic geometry, but now in the noncommutative setting following Artin, Zhang and Verevkin's ideas (e.g., \cite{CassidyVancliff2010, Lezama2020, Lezama2021, LezamaGomez2019, Rosenberg1995, Smith1991, Vancliff2015, VancliffVanRompay2000, Veerapen2017, ZhangZhang2008} and references therein).

\medskip

On the other hand, Manin \cite{Manin1991} commented the failure of attempts to obtain a noncommutative scheme theory \`a la Grothendieck for quantized algebras. Nevertheless, Van Oystaeyen and Willaert \cite{VanOystaeyenWillaert1995} studied this {\rm Proj} by developing a kind of scheme theory similar to the commutative theory. They noticed that this theory is possible only if the connected and $\mathbb{N}$-graded algebra considered contains \textquotedblleft enough\textquotedblright\ Ore sets. Algebras satisfying this condition are called {\em schematic}. They constructed a {\em generalized Grothendieck topology} for the free monoid on all Ore sets of a schematic algebra $R$, and defined a {\em noncommutative site} (c.f. \cite{VanOystaeyenWillaert1996a}) as a category with coverings on which sheaves can be defined, and formulated the Serre's theorem. With this, it was developed a sheaf theory which is similar to the scheme theory for commutative algebras. As a consequence of their treatment, an equivalence between the category of all coherent sheaves and the category ${\rm Proj}\ R$ was obtained in the sense of Artin \cite{Artin1992}. Some years later, Van Oystaeyen and Willaert \cite{VanOystaeyenWillaert1996a, vanOystaeyenWillaert1996, VanOystaeyenWillaert1997} presented a sequel of \cite{VanOystaeyenWillaert1995} in which they studied the cohomology of these algebras and proved a lifting property for Ore sets. This allowed to present many examples of schematic algebras like homogenizations of almost commutative algebras, Rees rings of universal enveloping algebras of Lie algebras, and three-dimensional Sklyanin algebras. A detailed treatment about schematic algebras can be found in Van Oystaeyen's book \cite{VanOystaeyen2000}.

\medskip

A few years ago, Lezama and Latorre \cite{LezamaLatorre2017} introduced the {\em semi-graded rings} with the aim of generalizing the $\mathbb{N}$-graded rings, the finitely $\mathbb{N}$-graded algebras and several algebras appearing in ring theory and noncommutative algebraic geometry which are not $\mathbb{N}$-graded algebras in a non-trivial sense. In that paper, they investigated some geometrical properties of semi-graded rings, within which is the Serre-Artin-Zhang-Verevkin theorem following Artin, Zhang and Verevkin's ideas (see also \cite{Lezama2021}). In this way, having in mind Van Oystaeyen and Willaert's ideas developed in \cite{VanOystaeyenWillaert1995} about a scheme theory for Proj in the setting of $\mathbb{N}$-graded algebras, it is natural to ask by a noncommutative scheme theory for semi-graded rings, and hence to investigate the {\em schematicness} of these objects in a more general context than $\mathbb{N}$-graded rings. This is the purpose of the paper. As expected, we generalize the results established in \cite{VanOystaeyenWillaert1995} for $\mathbb{N}$-graded algebras to the semi-graded setting (as a matter of fact, we do not impose the condition of connectedness of the algebra), and present another approach (Examples \ref{firstindep} and \ref{secondindep} show that the theory presented by Lezama \cite{Lezama2021, LezamaLatorre2017} and the one developed in this paper are independent) to the Serre-Artin-Zhang-Verevkin theorem for semi-graded rings which is a fundamental problem proposed for these objects  
\cite[Section 1.4 Problem 1]{Lezama2021}. 

\medskip

The article is organized as follows. In Section \ref{schematicalgebrasmotivation}, we recall some key facts about torsion theory, Serre's theorem in the commutative case and the noncommutative setting of $\mathbb{N}$-graded rings following the ideas presented by Artin and Zhang \cite{ArtinZhang1994}, and Van Oystaeyen and Willaert \cite{VanOystaeyenWillaert1995, VanOystaeyenWillaert1997}. Next, in Section \ref{localizationsemi-gradedrings}, we consider some preliminaries about semi-graded rings and semi-graded modules, and present some facts about the localization of these objects. In Section \ref{schematicness}, we formulate the notion of {\em schematicness} (Definition \ref{def.schematic}) for semi-graded rings without the assumption of connectedness established by Van Oystaeyen and Willaert for $\mathbb{N}$-graded rings. Section \ref{section-serre-theorem} contains the definition of noncommutative site with the aim of establishing the Serre-Artin-Zhang-Verevkin theorem in the semi-graded setting (Theorem \ref{Serre-theo}). Our results generalize those corresponding in the case of $\mathbb{N}$-graded rings (Remark \ref{we-generalize}) and allow to guarantee that other non-$\mathbb{N}$-graded algebras (even not connected) to be schematic (Example \ref{examplesillustrating}). As we said above, 
Examples \ref{firstindep} and \ref{secondindep} show that the theory presented by Lezama about Serre-Artin-Zhang-Verevkin theorem and the one developed in this paper are independent. Finally, in Section \ref{conclusionsfuturework}, we present some ideas for a future work that are motivated by different topics concerning schematic algebras \cite{VanOystaeyen2000, vanOystaeyenWillaert1996, VanOystaeyenWillaert1997, Willaert1998}.

\medskip

Throughout the paper, the term ring means an associative ring with identity not necessarily commutative. The letter $\Bbbk$ denotes an arbitrary field, and all algebras are $\Bbbk$-algebras. The symbols $\mathbb{N}$ and $\mathbb{Z}$ denote the set of natural numbers including zero, and the ring of integer numbers, respectively. For a ring $R$ and a subset $I$ of $R$, $I\vartriangleleft_l R$ means that $I$ is a left ideal of $R$. $Z(R)$ denotes the center of $R$, while the category of left $R$-modules is written as $R-{\rm Mod}$.

\section{Serre's theorem and schematic algebras}\label{schematicalgebrasmotivation}

We recall briefly some notions of algebraic geometry which are key in the proof of {\em Serre's theorem}.

\medskip

Following Hartshorne \cite{Hartshorne1977}, if $C = \Bbbk \oplus C_1 \oplus C_2 \oplus \dotsb$ is a positively graded commutative Noetherian ring generated in degree one, consider $Y = {\rm Proj}\ C$ and $Y(f) = \{\mathfrak{p}\in Y\mid f \notin \mathfrak{p}\}$, the Zariski open set corresponding to a homogeneous element $f\in C$. It is well-known that there is a finite subset $\{f_i\mid f_i\in C_1\}$ such that $Y = \bigcup_i Y(f_i)$. Equivalently, for every choice of $d_i\in \mathbb{N}$, there exists $n\in \mathbb{N}$ with $(C_{+})^n \subseteq \sum_i Cf_i^{d_i}$. In this way, for any finitely generated graded $C$-module $M$ we have
\begin{align*}
    \Gamma_*(M) := &\ \bigoplus_{n\in \mathbb{Z}} \Gamma(Y, \widetilde{M(n)})\\
    = &\ Q_{\kappa_{+}} (M) =  \lim_{{\overleftarrow{\ \ i\ \ }}} Q_{f_i}(M),
\end{align*}

where $\widetilde{M(n)}$ denotes the sheaf of modules associated to the shifted module $M(n)$ and $Q_{f_i}(M)$ is the localization of $M$ at $\{1, f_i, f_i^2,\dotsc \}$. Of course,
\[
Q_f(M) = \lim_{{\overleftarrow{\ \ i\ \ }}} Q_{ff_i}(M),
\]
where the inverse systems are defined as $q\le h$ if and only if $Y(g) \subseteq Y(h)$. This is precisely the key fact to prove {\em Serre's theorem}: the category of coherent $\mathcal{O}_Y$-modules is equivalent with a certain quotient category. 

\medskip 

In the noncommutative setting, for a {\em noncommutative positively graded Noetherian} $\Bbbk$-algebra $R = \Bbbk \oplus R_1 \oplus R_2 \oplus \dotsb $ with $R = \Bbbk[R_1]$ (notice that $R$ is connected, that is, $R_0 = \Bbbk$), Van Oystaeyen and Willaert \cite{VanOystaeyenWillaert1995} presented his interpretation of Serre's theorem for algebras with enough Ore sets called {\em schematic algebras}. With the aim of presenting the key ideas developed by them, we start by recalling some notions of torsion theory that we will use freely throughout the paper. For more details, we refer to Goldman \cite{Goldman}, Stenstrom \cite{Stenstrom1975} or Van Oystaeyen \cite{VanOystaeyen1978}.

\begin{definition}[{\cite[Section 2]{VanOystaeyenWillaert1995}}]
   Let $\mathscr{L}$ be a set of left ideals of an arbitrary ring $R$. $\mathscr{L}$ is said to be {\em a filter} if it satisfies the following conditions:
    \begin{enumerate}
        \item [$T_1$:] If $I\in \mathscr{L}$ and $I\subseteq J$, then $J\in \mathscr{L}$.
        \item [$T_2$:] If $I,J\in \mathscr{L}$, then $I\cap J\in \mathscr{L}$.
        \item [$T_3$:] If $I\in \mathscr{L}$ and $a\in R$, then $(I:a) := \{r\in R \mid ra\in I\}\in \mathscr{L}$.
        \end{enumerate} 
\end{definition}
      
The functor $\kappa:R-{\rm Mod}\to R-{\rm Mod}$ defined by
\[
\kappa(M) = \{m\in M\mid \ {\rm there\ exists}\ I\in \mathscr{L}\ {\rm with}\ Im = 0\},
\]
is a {\em left exact preradical}, that is, a left exact subfunctor of the identity functor on the category $R-{\rm Mod}$. A module $M$ satisfying $\kappa(M) = M$ is called a $\kappa$-{\em torsion module}, and if $\kappa(M) = 0$, then $M$ is said to be a $\kappa$-{\em torsion-free module}. It is straightforward to see that the family of torsion modules are closed under quotient objects and coproducts, while the torsion-free modules are closed under subobjects and products.
        
\medskip

The filter $\mathscr{L}$ is called {\em idempotent} (also called a {\em Gabriel topology}) when it satisfies the following condition:
\begin{enumerate}
\item [$T_4$:] If $I\vartriangleleft_l R$ and there exists $J\in \mathscr{L}$ such that for all $a \in J$ the relation $(I:a)\in \mathscr{L}$ holds, then $I\in \mathscr{L}$. 
\end{enumerate}
    
Condition $T_4$ implies that $\mathscr{L}$ is closed under products and that the functor $\kappa$ is a radical, that is, $\kappa(M / \kappa(M)) = 0$, for all $M\in R-{\rm Mod}$. 

\begin{proposition}\label{Noetherianidempotentfilter}
    If $R$ is a left Noetherian ring and $J_1\supseteq  J_2 \supseteq \dotsb$ is a descending chain of two-sided ideals of $R$, then the set 
    \[
    \mathscr{A}=\{I\vartriangleleft_l R \mid  \ {\rm there\ exist\ elements}\ n,m \in \mathbb{N}\ {\rm with}\ (J_m)^n\subseteq I\},
    \]
is an idempotent filter.
\begin{proof}
        \begin{enumerate}
            \item [\rm $T_1$:] If $(J_m)^n\subseteq I$ and $I\subseteq I'$, then it is clear that $(J_m)^n\subseteq I'$.
            \item [\rm $T_2$:] If $I, I'\in \mathscr{A}$, then there exist elements $n_1,n_2,m_1,m_2\in \mathbb{N}$ such that $(J_{m_1})^{n_1}\subseteq I$ and $(J_{m_2})^{n_2}\subseteq I'$. If we consider $n := {\rm max}\{n_1,n_2\}$ and $m := {\rm max}\{m_1,m_2\}$, then it follows that $(J_m)^n\subseteq I\cap I'$.
            \item [\rm $T_3$:] If $I\in \mathscr{A}$, then there exist elements $n,m\in\mathbb{N}$ such that $(J_m)^n\subseteq I$. Fix $a\in R$ and let $r\in (J_m)^n$. Since $(J_m)^n$ is an ideal of $R$, then $ra\in (J_m)^n\subseteq I$,  and so $r\in (I:a)$, that is, $(J_m)^n\subseteq (I:a)$.
            \item [\rm $T_4$:] Let $I\vartriangleleft_l R$ and $J \in \mathscr{A}$ such that for all $a\in J$, $(I:a)\in \mathscr{A}$. Since $J\in \mathscr{A}$, there exist elements $n,m\in \mathbb{N}$ with $(J_m)^n\subseteq J$. By assumption $R$ is left Noetherian, so $(J_m)^n$ is finitely generated by some elements $a_1,\dots,a_l$. Notice that $a_i \in J$, for $1\leq i\leq l$, whence $(I:a_i)\in \mathscr{A}$. In this way, there exist elements $k_i,j_i\in \mathbb{N}$ such that $(J_{j_i})^{k_i}\subseteq (I:a_i)$. If $k := {\rm max}\{k_i\}_{1\le i\le l}$ and $j := {\rm max}\{j_i\}_{1\le i \le l}$, then $(J_j)^k\subseteq (I:a_i)$, for every $1\leq i\leq l$.

            Let $r\in (J_j)^k$ and $s\in (J_m)^n$. There exist elements $r_1,\dots, r_l\in R$ such that $s=\sum_{i=1}^lr_ia_i$, and so  $rs=\sum_{i=1}^lrr_ia_i$. Since $(J_j)^k$ is an ideal of $R$, we have that $rr_i\in (J_j)^k$, for all $i$. Thus $rr_ia_i\in I$, whence $rs\in I$. It follows that $(J_{j+m})^{k+n}\subseteq (J_j)^k(J_m)^n\subseteq I$.
        \end{enumerate}
    \end{proof}
\end{proposition}

Consider a ring $R$, $\mathscr{L}$ an idempotent filter of left ideals of $R$ and its associated radical $\kappa$. For an $R$-module $M$, we recall the quotient module $Q_\kappa(M)$ of $M$ (Definition \ref{quotientmoduleMwrtQ}). With this aim, we introduce the Definition \ref{relationpairs}.

\begin{definition}\label{relationpairs}
    Let $M\in R-{\rm Mod}$. Consider the family $\Omega_M$ of pairs $(I,f)$ with $I\in \mathscr{L}$ and $f:I\rightarrow M$ an $R$-homomorphism. We define the relation $\sim$ on $\Omega_M$ as $(I_1,f_1) \sim (I_2,f_2)$ if and only if there exists an element $J\in \mathscr{L}$ such that $J\subseteq I_1\cap I_2$ and $f|_{J}=g|_{J}$.
\end{definition}

It is straightforward to see that $\sim$ is an equivalence relation. The equivalence class of the element $(I,f)$ is denoted as $[I,f]$, and the set of equivalence classes will be written as $M_{\mathscr{L}}$. For two elements $[I,f], [J,g]\in M_{\mathscr{L}}$, we define their sum as $[I,f]+[J,g] = [I\cap J,f+g]$. It is easy to see that this sum is well defined and that $(M_{\mathscr{L}}, +)$ is an Abelian group.

\medskip

It is also easy to see that if $I, J\in\mathscr{L}$ and $f\in {\rm Hom}(I,R)$, then $f^{-1}(J)\in \mathscr{L}$. In this way, when we take elements $[I,f]\in R_\mathscr{L}$ and $[J,g]\in M_\mathscr{L}$, we can define the product of these elements as $[I,f]\cdot [J,g] = [f^{-1}(J),g\circ f]$. Notice that this product is well defined, and so $R_\mathscr{L}$ is actually a ring with identity $[R,{\rm id}_R]$. Thus, $M_\mathscr{L}$ is a left $R_\mathscr{L}$-module.

\medskip

Let $m\in M$. We define the application $\beta(m):R\rightarrow M$ given by $\beta(m)(r)=rm$. It is well-known that $\beta:M\rightarrow {\rm Hom}(R,M)$ is an isomorphism of $R$-modules. If we consider $\varphi_M:M\rightarrow M_\mathscr{L}$ defined by  $\varphi_M(m)=[R,\beta(m)]$, then it follows that $\varphi_R$ is a ring homomorphism, and so we can consider $R_\mathscr{L}$ and $M_\mathscr{L}$ as $R$-modules with the action given by $r[I,f] := [R,\beta(r)][I,f]$. Note that $\varphi_M$ is actually a homomorphism of $R$-modules. Since ${\rm Ker}(\varphi_M)=\kappa(M)$, the fundamental isomorphism theorem guarantees that $\varphi_M(M)\cong M/\kappa(M)$. In this way, if $\kappa(M) = 0$ then we can embed $M$ into $M_\mathscr{L}$. 

\medskip

For an element $\xi\in M_\mathscr{L}$ given by $\xi = [I,f]$, and an element $a\in I$, notice that  $a\xi=[R,\beta(f(a))]=\varphi_M(f(a))$, which shows that $I\xi\subseteq \varphi_M(M)$, that is, ${\rm Coker}(\varphi_M)$ is a torsion module in $\mathscr{L}$.  

\medskip

Considering the notation and terminology above, we present the definition of the quotient module of an object  $M$ in $R-{\rm Mod}$.

\begin{definition}\label{quotientmoduleMwrtQ}
The {\em quotient module of} $M$ with respect to $\kappa$ is defined as $Q_\kappa(M)=(M/\kappa(M))_\mathscr{L}$. Since $\mathscr{L}$ is idempotent, it follows that $\kappa\left(M/\kappa(M) \right)=0$. Hence, we can embed $M/\kappa(M)$ into $Q_\kappa(M)$.
\end{definition}

Equivalently, the {\em quotient module} of $M$ with respect to $\kappa$ is given by
\[
Q_{\kappa}(M) = \varinjlim_{I\in \mathscr{L}} {\rm Hom}_R (I, M/\kappa(M)).
\]
$Q_{\kappa}(M)$ turns out to be a module over the ring $Q_{\kappa}(R)$.

\medskip

Following \cite{Stenstrom1975}, recall that an $R$-module $E$ is $\kappa$-{\em injective} if for every $R$-module $M$ and each submodule $N$ such that $\kappa(M/N) = M/N$, every $R$-homomorphism $N\rightarrow E$ can be extended to an $R$-homomorphism $M\rightarrow E$. We say that $E$ is $\kappa$-{\em closed} (also known as {\em faithfully}  $\kappa$-{\em injective}) if the extension of the homomorphism is unique. It is straightforward to see that $E$ is $\kappa$-closed if and only if $E$ is $\kappa$-injective and $\kappa$-torsion-free. By using these notions, we can characterize $Q_\kappa(M)$ in the following way: $Q_\kappa(M)$ is the unique $\kappa$-closed module containing $N = M/\kappa(M)$ such that $Q_\kappa(M)/N$ is $\kappa$-torsion. 

 \begin{example}[{\cite[p. 111]{VanOystaeyenWillaert1995}}]
   \begin{enumerate}
       \item [(i)] Consider $S$ a left Ore set in an arbitrary ring $R$. The set
    \[
\mathscr{L}(S) = \{I\vartriangleleft_l R\mid I\cap S\neq \emptyset\}
    \]
    is an idempotent filter. If $\kappa_S$ denotes its corresponding radical and $Q_S(M)$ is the module of quotients of $M$, then it is straightforward to see that $Q_S(M)$ is isomorphic to $S^{-1}M$, i.e., the classical Ore localization of $M$ at $S$.

    \item [(ii)] If $R=\bigoplus_{k\ge 0}R_k$ is a positively graded Noetherian ring and $R_+$ denotes the two-sided ideal $\bigoplus_{k\ge 1}R_k$, by Proposition \ref{Noetherianidempotentfilter}, the set
    \[ 
\mathscr{L}(\kappa_+) = \{I\vartriangleleft_l R\mid {\rm there\ exists}\ n\in \mathbb{N}\ {\rm with}\ (R_{+})^n\subseteq I\}
    \]
    is an idempotent filter. The corresponding radical is denoted by $\kappa_+$.
   \end{enumerate}
\end{example}

From the treatment above, having in mind that the filter $\mathscr{L}(\kappa_{+})$ is idempotent, Van Oystaeyen and Willaert \cite{VanOystaeyenWillaert1995} formed the {\em quotient category} $(R, \kappa_{+})$-gr, that is, the full subcategory of $Q_{\kappa_{+}}(R)$-gr consisting of modules of the form $Q_{{\kappa}_{+}}(M)$ for some graded $R$-module $M$. Notice that $(R, \kappa_{+})$-gr is equivalent to the full subcategory of 
$R$-gr consisting of the $\kappa_{+}$-closed modules and define ${\rm Proj}\ R$ as the Noetherian objects in  $(R, \kappa_{+})$-gr. Since they wanted to describe the objects of ${\rm Proj}\ R$ by means of objects of usual module categories in the same way as for commutative algebras, they need modules determined by Ore localizations. This is the content of the following definition.

%\medskip

%There are three equivalent ways to define ${\rm Proj}\ R$ \cite[Section 3]{VanOystaeyenWillaert1995}:
%\begin{enumerate}
%    \item [\rm (i)] The Noetherian objects in $(R, \kappa_{+})$-gr, i.e., those objects of $(R, \kappa_{+})$-gr that satisfy the ascending chain condition on subobjects.
%    \item [\rm (ii)] The full subcategory of $Q_{\kappa_{+}}(R)$-gr consisting of all modules of the form $Q_{\kappa_+}(M)$ for some finitely generated $R$-module $M$.
%    \item [\rm (iii)] The full subcategory of $R$-gr consisting of all $\kappa_{+}$-closed modules which are torsion over a finitely generated $R$-submodule.
%\end{enumerate}

%Van Oystaeyen and Villaert \cite{VanOystaeyenWillaert1995} presented a description of the objects of ${\rm Proj}\ R$ by using objects of module categories in the same way as for commutative objects, that is, they considered modules to be determined by Ore localizations. This fact suggested the definition of {\em schematic algebra}. Recall that a {\em non-trivial} ({\em left}) {\em Ore set} of $R$ is a homogeneous (left) Ore set $S$ of $R$ such that $1\in S,\ 0\notin S$, and $S\not\subseteq R_0$. For further details about localization in Noetherian rings, see Jategaonkar's  excellent book \cite{Jategaonkar1986}.

\begin{definition}[{\cite[Definition 1]{VanOystaeyenWillaert1995}}]\label{schematicgradedsetting}
    The noncommutative positively graded Noetherian $\Bbbk$-algebra $R = \Bbbk \oplus R_1 \oplus R_2 \oplus \dotsb $ with $R = \Bbbk[R_1]$ is {\em schematic} if there is a finite set $I$ of homogeneous left Ore sets of $R$ such that for every $S\in I,\ S\cap R_+\neq \emptyset$, and such that one of the following equivalent properties is satisfied:
    \begin{enumerate}
        \item [\rm (i)] for each $(r_S)_{S\in I}\in \prod_{S\in I}S$, there exists $m\in \mathbb{N}$ such that $(R_{+})^{m}\subseteq \sum_{S\in I}Rr_S$,
        \item [\rm (ii)] $\bigcap_{S\in I}\ \mathscr{L}(S) = \mathscr{L}(\kappa_{+})$,
        \item [\rm (iii)] $\bigcap_{S\in I}\kappa_S(M) = \kappa_{+}(M)$, for all $M\in R-{\rm Mod}$,
        \item [\rm (iv)] $\bigwedge_{S\in I}\kappa_S = \kappa_{+}$ where $\bigwedge$ denotes the infimum of torsion theories.
        \end{enumerate}
        \end{definition}

In \cite {VanOystaeyenWillaert1995, VanOystaeyenWillaert1996a}, Van Oystaeyen and Willaert constructed the noncommutative site, a category with coverings on which sheaves can be defined, and formulated the Serre's theorem. Examples \ref{almostcommutativealgebras} and \ref{VanOystaeyenWillaert1997Example2} contain remarkable examples of schematic algebras.

\begin{example}\label{almostcommutativealgebras}
Recall that if $R$ is a {\em positively} filtered $\Bbbk$-algebra by the family $(F_nR)_{n\ge 0}$ (i.e., $F_0R = \Bbbk$), $\sigma: R\to G(R)$ is the principal symbol map, and $\widehat{R}$ is the Rees-ring of $R$, it is well-known that $G(R)$ and $\widehat{R}$ are positively graded and there is a canonical central element $X$ in $\widehat{R}$ of degree $1$ such that $\widehat{R} / \langle X\rangle \cong G(R)$. If $\widehat{R}$ is Noetherian, this is equivalent to $G(R)$ being Noetherian or the filtration of $R$ being Zariskian. Notice that if $S$ is a multiplicatively set in $R$ such that $\sigma(S)$ is a multiplicative set in $G(R)$, then $\widehat{R} =  \{sX^{{\rm deg} \sigma(S)}\mid s\in S\}$ is a multiplicative set (consisting of homogeneous elements) in $\widehat{R}$. For more details about graded rings, Zariskian filtrations and Rees rings, see Li and Van Oystaeyen's book  \cite{LiVanOystaeyen1996}.

\medskip

For $R$ positively filtered by $(F_nR)_{n\ge 0}$, if $G(R)$ is schematic then $\widehat{R}$ is schematic \cite[Theorem 1]{VanOystaeyenWillaert1997}. In this way, since for an almost commutative ring $R$ there exists a filtration on $R$ such that $G(R)$ is commutative, it follows that its Rees-ring is schematic. For example, the algebra $R$ generated by three elements $x, y$ and $z$ of degree 1 with relations $xy - yx = z^2,\ xz - zx = 0$, and $yz - zy = 0$, is schematic since it is the Rees-ring of the first Weyl algebra $A_1(\Bbbk)$ with respect to the Bernstein-filtration (this algebra is known as the {\em homogenized Weyl algebra}). 
\end{example}

Van Oystaeyen and Willaert \cite[p. 199]{VanOystaeyenWillaert1997} said that \textquotedblleft it is probably not true that the class of schematic algebras is closed under iterated Ore extensions since Ore sets in a ring $R$ need not be Ore in an Ore extension $R[x;\sigma, \delta]$\textquotedblright. Nevertheless, the following proposition shows that under suitable conditions, these extensions are schematic.

\begin{proposition}[{\cite[Theorem 3]{VanOystaeyenWillaert1997}}]\label{VanOystaeyenWillaert1997Theorem3}
Given a positively graded ring $R$ which is generated by $R_1$ and which is schematic by means of Ore sets $S_i$, given $\sigma$ a graded automorphism of $R$ and $\delta$ a $\sigma$-derivation of degree $1$, then for all $s_i\in \prod S_i$ and for all $m\in \mathbb{N}$, there exists $p\in \mathbb{N}$ such that
\[
(R[x;\sigma, \delta]_{+})^{p} \subseteq M := \sum_{i} R[x;\sigma, \delta]s_i + R[x;\sigma, \delta]x^m,
\]
where $R[x;\sigma, \delta]$ denotes the Ore extension considered with graduation $(Rx[x;\sigma, \delta])_n = \bigoplus_{k = 0}^{n} R_kx^{n-k}$.
\end{proposition}

Proposition \ref{VanOystaeyenWillaert1997Theorem3} is one of the results that Van Oystaeyen and Willaert \cite{VanOystaeyenWillaert1997} used  to show that the algebras in Example \ref{VanOystaeyenWillaert1997Example2} are schematic.  

\begin{example}[{\cite[Examples 2-5]{VanOystaeyenWillaert1997}}]\label{VanOystaeyenWillaert1997Example2}
\begin{itemize}
    \item [\rm (i)] The {\em coordinated ring of quantum $2\times2$-matrices} $\mathscr{O}_q(M_2(\mathbb{C}))$ with $q\in \mathbb{C}$ is generated by elements $a, b, c$ and $d$ subject to the relations
    \begin{align*}
        ba&=q^{-2}ab,  &  ca &= q^{-2}ac, &  bc&=cb,\\
        db&=q^{-2}bd,  &  dc&=q^{-2}cd,  &  ad-da&=(q^2-q^{-2})bc.
    \end{align*}
    \item [\rm (ii)] {\em Quantum Weyl algebras} $A_n^{\overline{q}, \Lambda}$ defined by Alev and Dumas \cite{AlevDumas1994} are given by an $n\times n$ matrix $\Lambda=(\lambda_{ij})$ with $\lambda_{ij}\in\Bbbk^*$ and a row vector $\overline{q}=(q_1,\dots,q_n)$, where $q_i\neq 0$ for every $i$, the algebra is generated by elements $x_1,\dots,x_n,y_1,\dots,y_n$ subject to relations ($i<j$) given by
    \begin{align*}
        x_ix_j & = \mu_{ij}x_jx_i, & x_iy_j & = \lambda_{ji}y_jx_i, & y_jy_i & = \lambda_{ji}y_iy_j,\\
        x_jy_i & = \mu_{ij}y_ix_j & x_jy_j & = 1+q_jy_jx_j+\sum_{i<j}(q_i-1)y_ix_i,
    \end{align*}
    where $\mu_{ij}=\lambda_{ij}q_i$.
    \item [\rm (iii)] {\em Three dimensional Sklyanin algebras} $A_{\Bbbk}$ {\em over a field} $\Bbbk$ according to Artin et al.  \cite{ArtinTateVandenBergh1990} are graded $\Bbbk$-algebras generated by three homogeneous elements $x, y$ and $z$ of degree 1 satisfying the relations
     \begin{equation*}
        axy+byx+cz^2=0, \quad ayz+bzy+cx^2=0, \quad {\rm and}\quad  azx+bxz+cy^2=0,
    \end{equation*}
      where $a,b,c\in \Bbbk$.
      \item [\rm (iv)] Color Lie super algebras defined by Rittenberg and Wyler \cite{Rittenberg}.
\end{itemize}
\end{example}

\begin{remark}
Of course, there are examples of non-schematic algebras. If we take the graded algebra $\Bbbk \{x, y\}/\langle yx - xy - x^2\rangle$ and suppose that ${\rm char}(\Bbbk) = 0$, then its subalgebra generated by $y$ and $xy$ is not left schematic \cite[p. 203]{VanOystaeyenWillaert1997}.
\end{remark}

\section{Semi-graded rings}\label{localizationsemi-gradedrings}

Lezama and Latorre \cite{LezamaLatorre2017} presented an introduction to noncommutative algebraic geometry for non-$\mathbb{N}$-graded algebras and finitely non-graded algebras by defining a new class of rings, the {\em semi-graded rings}. These rings extend several kinds of noncommutative rings of polynomial type such as Ore extensions \cite{Ore1931, Ore1933}, families of differential operators generalizing Weyl algebras and universal enveloping algebras of finite dimensional Lie algebras \cite{Bavula1992, BellGoodearl1988, Smith1991}, algebras appearing in mathematical physics \cite{IPR01, ReyesSarmiento2022, Zhedanov}, down-up algebras \cite{Benkart1999, BenkartRoby1998, Kirkmanetal1999}, ambiskew polynomial rings \cite{JordanWells1996, Jordan2000}, 3-dimensional skew polynomial rings \cite{BellSmith1990, Redman1999, ReyesSarmiento2022, Rosenberg1995}, PBW extensions \cite{BellGoodearl1988}, skew PBW extensions \cite{GallegoLezama2011}, and others. A detailed list of examples of semi-graded rings and its relationships with other algebras can be found in Fajardo et al. \cite{libropbw}. Ring-theoretical, algebraic and geometrical properties of semi-graded rings have been investigated in the literature by several authors (e.g., \cite{Artamonov2015, CalderonReyes2022, HashemiKhalilAlhevaz2017, Reyes2019, ReyesRodriguez2021, ReyesSuarez2020, SuarezChaconReyes2022, SuarezReyesSuarez2022, Tumwesigyeetal2020} and references therein).

\begin{definition}[{\cite[Definition 2.1]{LezamaLatorre2017}}]\label{def.SG2}
Let $R$ be a ring. $R$ is said to be {\em semi-graded} ($\mathsf{SG}$) if there exists a collection $\{R_n\}_{n\in\mathbb{Z}}$ of subgroups $R_n$ of the additive group $R^{+}$ such that the following conditions hold:
\begin{enumerate}
    \item [\rm (i)] $R=\bigoplus_{n\in\mathbb{Z}}R_n$.
    \item [\rm (ii)] For every $m,n\in\mathbb{Z}$, $R_mR_n\subseteq \bigoplus_{k\leq m+n} R_k$. 
    \item [\rm (iii)] $1\in R_0$.
\end{enumerate}
\end{definition}
The collection $\{R_n\}_{n\in\mathbb{Z}}$ is called {\em a semi-graduation of} $R$, and we say that the elements of $R_n$ are {\em homogeneous of degree} $n$.

\medskip

We say that $R$ is {\em positively semi-graded} if $R_n=0$ for every $n<0$. If $R$ and $S$ are semi-graded rings and $f: R\rightarrow S$ is a ring homomorphism, then we say that $f$ is {\em homogeneous} if $f(R_n)\subseteq S_n$ for every $n\in\mathbb{Z}$. 

\medskip

Definitions \ref{def.FSG} and \ref{def.FSGA} recall the notion of finitely semi-graded ring and finitely semi-graded algebra, respectively.

\begin{definition}[{\cite[Definition 2.4]{LezamaLatorre2017}}]\label{def.FSG}
A ring $R$ is called  \textit{finitely semi-graded} ($\mathsf{FSG}$) if it satisfies the following conditions:
\begin{enumerate}
    \item [\rm (i)] $R$ is $\mathsf{SG}$.
   \item [\rm (ii)] There exist finitely many elements $x_1,\dots,x_n\in R$ such that the subring generated by $R_0$ and $x_1,\dots,x_n$ coincides with $R$.
    \item [\rm (iii)] For every $n\geq0$, $R_n$ is a free $R_0$-module of finite dimension.
\end{enumerate}
\end{definition}

\begin{definition}[{\cite[Definition 10]{LezamaGomez2019}}]\label{def.FSGA}
A $\Bbbk$-algebra $R$ is said to be \textit{finitely semi-graded} ($\mathsf{FSG}$) if the following conditions hold:
\begin{enumerate}
    \item [\rm (i)] $R$ is an $\mathsf{FSG}$ ring with semi-graduation given by $R  = \bigoplus_{n\geq0}R_n$ .
    \item [\rm (ii)] For every $m,n\geq1$, $R_mR_n\subseteq  R_1\oplus \dots\oplus R_{m+n}$. 
   \item [\rm (iii)] $R$ is connected, i.e., $R_0=\Bbbk$.
    \item [\rm (iv)] $R$ is generated in degree 1.
\end{enumerate}
\end{definition}

From Definition \ref{def.FSGA}, it is straightforward to see that if $R$ is a $\Bbbk$-algebra, then $R_+:=\bigoplus_{n\geq1}R_n$ is a maximal ideal of $R$.

\medskip

Notice that $\mathbb{N}$-graded rings are $\mathsf{SG}$. Finitely graded $\Bbbk$-algebras, PBW extensions \cite{BellGoodearl1988}, 3-dimensional skew polynomial rings \cite{BellSmith1990}, down-up algebras \cite{Benkart1999, BenkartRoby1998}, diffusion algebras \cite{IPR01}, and skew PBW extensions \cite{GallegoLezama2011} are examples of $\mathsf{FSG}$ rings.

\begin{definition}[{\cite[Definition 2.2]{LezamaLatorre2017}}]
    Let $R$ be an $\mathsf{SG}$ ring and let $M$ be an $R$-module. We say that $M$ is  {\em semi-graded} if there exists a collection $\{M_n\}_{n\in\mathbb{Z}}$ of subgroups $M_n$ of the additive group $M^+$ such that the following conditions hold:
    \begin{enumerate}
        \item [\rm (i)] $M=\bigoplus_{n\in\mathbb{Z}}M_n$.
        \item [\rm (ii)] For every $m\geq 0$ and $n\in\mathbb{Z}$, $R_mM_n\subseteq \bigoplus_{k\leq m+n}M_k$.
    \end{enumerate}
\end{definition}
The collection $\{M_n\}_{n\in\mathbb{Z}}$ is called {\em a semi-graduation of} $M$, and we say that the elements of $M_n$ are {\em homogeneous of degree} $n$.

\medskip

$M$ is said to be {\em positively semi-graded} if $M_n=0$ for every $n<0$. Let $f: M\rightarrow N$ be a homomorphism of $R$-modules, where $M$ and $N$ are semi-graded $R$-modules. We say that $f$ is {\em homogeneous} if $f(M_n)\subseteq N_n$ for every $n\in\mathbb{Z}$. 

\begin{definition}[{\cite[Definition 2.3]{LezamaLatorre2017}}]
    Let $R$ be an $\mathsf{SG}$ ring, $M$ an $\mathsf{SG}$ $R$-module, and $N$ a submodule of $M$. We say that $N$ is a {\em semi-graded} ($\mathsf{SG}$) {\em submodule of} $M$ if $N=\bigoplus_{n\in\mathbb{Z}}N_n$ where $N_n=M_n\cap N$. In this case, $N$ is an SG $R$-module.
\end{definition}

\begin{proposition}[{\cite[Proposition 2.6]{LezamaLatorre2017}}]
    If $R$ is an $\mathsf{SG}$ ring, $M$ an $\mathsf{SG}$ $R$-module and $N$ a submodule of $M$, then the following conditions are equivalent:
    \begin{enumerate}
        \item [\rm (1)] $N$ is a semi-graded submodule of $M$.
        \item [\rm (2)] For every $z\in N$, the homogeneous components of $z$ are in $N$.
        \item [\rm (3)] $M/N$ is an $\mathsf{SG}$ $R$-module with semi-graduation given by
        \[
        (M/N)_n=(M_n+N)/N,\ n\in\mathbb{Z}.
        \]
    \end{enumerate}
\end{proposition}
% \begin{proof}
%The equivalence between (1) and (2) it is clear.

% ${\rm (2)} \implies {\rm (3)}$ Let $\overline{M}_n:= (M_n+N)/N$, with $n\in\mathbb{Z}$, and $\overline{M}=M/N$. Consider $z\in M$ given by $z = z_1+\dots+z_l$, where $z_k\in M_{n_k}$ for $1\leq k\leq l$. Then $\overline{z}=\overline{z_1}+\dots+\overline{z_l}\in \sum_{n\in\mathbb{Z}}\overline{M}_n$, and so $\overline{M}=\sum_{n\in\mathbb{Z}}\overline{M}_n$. Since the equality $\overline{z_1}+\dots+\overline{z_l}=\overline{0}$ implies $z_1+\dots+z_l\in N$, part (2) shows that $z_k\in N$, i.e., $\overline{z_k} = \overline{0}$ for every $1\leq k\leq l$, and hence the sum is direct. 
 
% Now, let $b_m\in R_m$ and $\overline{z_n}\in\overline{M}_n$. Then $b_m\overline{z_n}=\overline{b_mz_n}=\overline{d_1+\dots+d_p}$, with $d_i\in M_{n_i}$ and $n_i\leq n+m$, whence $\overline{b_m}\overline{z_n}=\overline{d_1}+\dots+\overline{d_p}\in\bigoplus_{k\leq m+n}\overline{M}_k$. Therefore we have showed that $\overline{M}$ is semi-graded.
 
% ${\rm (3)} \implies {\rm (2)}$ Let $z = z_1+\dotsb + z_l\in N$, with $z_i\in M_{n_i}$, and $1\leq i\leq l$. Then $\overline{0} = \overline{z_1}+\dotsb + \overline{z_l}\in\bigoplus_{n\in\mathbb{Z}}\overline{M}_n$, which implies that  $\overline{z_i} = \overline{0}$, and so $z_i\in N$, for every $i$.
% \end{proof}

If $M$ is an $\mathsf{SG}$ $R$-module and $\{N_i\}_{i\in I}$ is a family of $\mathsf{SG}$ submodules of $M$, then it is clear that $\bigcap_{i\in I} N_i$ is an $\mathsf{SG}$ submodule of $M$. 

\medskip

Let $X$ be a subset of $M$. We define the $\mathsf{SG}$ submodule generated by $X$ as the intersection of all $\mathsf{SG}$  submodules containing $X$, and we will denote it as $\langle X\rangle^{\mathsf{SG}}$. If $X=\{x_1,\dots,x_n\}$, then write $\langle X\rangle^{\mathsf{SG}} = \langle x_1, \dotsc,  x_n\rangle^{\mathsf{SG}}$. We will say that $M$ is a {\em finitely generated} $\mathsf{SG}$ $R$-module if there exist finitely elements $m_1,\dots,m_n$ such that $M = \langle m_1, \dotsc, m_n\rangle^\mathsf{SG}$. If $M$ is simultaneously a module over different kinds of rings and there is risk of confusion, we write $\langle*\rangle^{\mathsf{SG}}_R$ to indicate the ring $R$ we are considering. 

\medskip

In a similar way, if $R$ is a positively $\mathsf{SG}$ ring, for $t\in \mathbb{N}$, we define $R_{\ge t}$ as the intersection of all two-sided ideals that are $\mathsf{SG}$ submodules containing $\bigoplus_{k\ge t}R_k$. 

\begin{remark}
If $R$ is a positively $\mathsf{SG}$ left Noetherian ring, then Proposition \ref{Noetherianidempotentfilter} shows that 
\[
\mathscr{L}(\kappa_{+}) = \{I\vartriangleleft_l R\mid {\rm there\ exist}\ n,m\in \mathbb{N}\ {\rm with}\ (R_{\ge m})^n\subseteq I\}
 \]
is an idempotent filter. The corresponding left exact radical is denoted by $\kappa_{+}$ and $Q_{\kappa_+}(M)$ is the module of quotients of $M$.
\end{remark}

%Now, we recall briefly the Serre-Artin-Zhang-Verevkin theorem formulated by Lezama \cite{Lezama2021, LezamaLatorre2017} in the setting of semi-graded rings.

%\medskip

%In \cite{LezamaLatorre2017} the theorem was formulated for semi-graded left Noetherian rings that are domains. However, the theorem for finitely graded algebras does not include this assumption

Next, we want to formalize several constructions concerning semi-graded rings which are necessary to formulate Serre's theorem.

\subsection{Localization of semi-graded rings} 

With the aim of defining {\em good} Ore sets (Definition \ref{goodOreset}), for $R$ an $\mathsf{SG}$ ring and an element $n\in \mathbb{Z}$, we consider the following sets: 
\begin{align*}
R'_n = &\ \{r\in R_n\mid \ {\rm for\ all}\ m\in\mathbb{Z},\ {\rm and}\ {\rm for\ all}\ h\in R_m, rh\in R_{n+m}\},\\
R''_n = &\ \{r\in R'_n\mid \ {\rm for\ all}\ m\in\mathbb{Z},\ {\rm and}\ {\rm for\ all}\ h\in R_m, hr\in R_{n+m}\},\\
R' = &\ \bigcup_{n\in\mathbb{Z}}R'_n,\\
R'' = &\ \bigcup_{n\in\mathbb{Z}}R''_n.
\end{align*}

\begin{definition}\label{goodOreset}
Let $R$ be an $\mathsf{SG}$ ring and consider a left Ore set $S$ of $R$. We say that $S$ is {\em good} if the following conditions hold:
    \begin{enumerate}
        \item [\rm (i)] $S\subseteq R''$, and 
        \item [\rm (ii)] if $s\in S$ and $r\in R'$, then there exist elements $u\in R'$ and $v\in S$ such that $us=vr$.
    \end{enumerate}
\end{definition}

From Definition \ref{goodOreset} it follows that for any elements $s_1,\dots,s_k\in S$, there exist $r_1,\dots,r_k\in R'$ such that $r_is_i=r_js_j\in S$, for every $i,j$.

\begin{definition}
    Let $R$ be an $\mathsf{SG}$ ring and $M$ an $\mathsf{SG}$ $R$-module. We say that $M$ is {\em localizable semi-graded} ($\mathsf{LSG}$) if for every element $(n,m)\in \mathbb{Z}^2$, the inclusion $R'_nM_m\subseteq M_{n+m}$ holds. 
\end{definition}

\begin{proposition}\label{localization-module}
Let $R$ be an $\mathsf{SG}$ ring, $S$ a good left Ore set, and $M$ an $\mathsf{LSG}$ $R$-module. Then $S^{-1}M$ is an $\mathsf{LSG}$ $R$-module with semi-graduation given by
\[
    (S^{-1}M)_n = \left\{\frac{f}{s}\mid f\in \bigcup_{k\in\mathbb{Z}} M_k,\ {\rm deg}(f) - {\rm deg}(s)=n \right\}.
\]
\end{proposition}
\begin{proof}
First of all, let us show that $(S^{-1}M)_n$ is a subgroup of $S^{-1}M$. It is clear that $0=\frac{0}{1}\in (S^{-1}M)_n$ and that $(S^{-1}M)_n$ have additive inverses. Consider elements  $\frac{p}{s},\frac{q}{t}\in (S^{-1}M)_n$. Then  ${\rm deg}(p)-{\rm deg}(s) = {\rm deg}(q)-{\rm deg}(t)=n$. There exist elements $u\in R'$ and $v\in S$ such that $us=vt\in S$. Note that ${\rm deg}(u) + {\rm deg}(s) = {\rm deg}(v) + {\rm deg}(t)$. Since $u,v\in R'$, it follows that $up$ and $vq$ are homogeneous elements satisfying ${\rm deg}(up) = {\rm deg}(u) + {\rm deg}(p) = {\rm deg}(v) + {\rm deg}(q) = {\rm deg}(vq)$, whence  $\frac{p}{s}+\frac{q}{t}=\frac{up+vq}{vt}$ is a homogeneous elements of degree ${\rm deg}(v) + {\rm deg}(q)-({\rm deg}(v)+{\rm deg}(t))=n$.
    
\medskip

Since it is clear that $S^{-1}M$ is the sum of the subgroups $(S^{-1}M)_n$, let us show that the sum is direct. Consider the sum 
\[
    \sum_{i=1}^k\frac{m_i}{s_i}=0
\]  
of homogeneous elements of  $S^{-1}M$ with different degrees, that is, ${\rm deg}(m_i) - {\rm deg}(s_i)\neq {\rm deg}(m_j) -{\rm deg}(s_j)$, for $i\neq j$. There exist elements $r_1,\dots,r_k\in R'$ such that $r_is_i=r_js_j$ for all $i,j$, which implies that 
\[
    0= \sum_{i=1}^k\frac{m_i}{s_i}= \frac{\sum_{i=1}^kr_im_i}{r_1s_1}.
\]
Hence, there exists an element $s\in S$ such that $0=s\sum_{i=1}^kr_im_i=\sum_{i=1}^ksr_im_i$. Since $s,r_i\in R'$ and $m_i$ is homogeneous for $0\leq i \leq k$, then every one of the terms above is homogeneous. Using that $r_is_i=r_js_j$, we have ${\rm deg}(r_i) + {\rm deg}(s_i) = {\rm deg}(r_j) + {\rm deg}(s_j)$, whence ${\rm deg}(s) + {\rm deg}(r_i) + {\rm deg}(m_i)\neq {\rm deg}(s) + {\rm deg}(r_j) + {\rm deg}(m_j)$, which shows that $sr_im_i=0$. Thus, $0=\frac{r_im_i}{r_is_i}=\frac{m_i}{s_i}$.

\medskip

Now, let us see that $R_a(S^{-1}M)_b\subseteq\bigoplus_{k\leq a+b}(S^{-1}M)_k$. Let $r\in R_a$ and $\frac{m}{s}\in(S^{-1}M)_b$. There exist elements $r'\in R$ and $s'\in S$ such that $r's=s'r$. Since $s, s'\in R''$ and $r$ is homogeneous, we can take the element $r'$ being homogeneous. Then ${\rm deg} (r') = {\rm deg}(s') + {\rm deg}(r) - {\rm deg}(s)$, and using that $r\frac{m}{s} = \frac{r'm}{s'}$ and $r'm\in\bigoplus_{k\leq {\rm deg}(r') + {\rm deg}(m)}M_k$, it follows that $\frac{r'm}{s'}\in \bigoplus_{k\leq {\rm deg}(r') + {\rm deg}(m) - {\rm deg}(s')}(S^{-1}M)_k$. Since ${\rm deg}(r') + {\rm deg}(m)- {\rm deg}(s') = {\rm deg}(r) + {\rm deg}(m) - {\rm deg} (s) = a+b$, then $r\frac{m}{s}\in \bigoplus_{k\leq a+b}(S^{-1}M)_k$. This fact proves that $S^{-1}M$ is an $\mathsf{SG}$ $R$-module.

\medskip

Notice that if we above consider the element $r\in R'_a$, then we can take $r'\in R'$, whence $r'm\in M_{{\rm deg}(r') + {\rm deg}(m)}$, and so $r\frac{m}{s}\in(S^{-1}M)_{a+b}$. This shows that $S^{-1}M$ is an $\mathsf{LSG}$ $R$-module.
\end{proof}

The next result shows that the localization of an $\mathsf{SG}$ ring by considering a good Ore set is again an $\mathsf{SG}$ ring. 

\begin{proposition}
    Let $R$ be an $\mathsf{SG}$ ring and $S$ a good left Ore set. Then $S^{-1}R$ is an $\mathsf{SG}$ ring with semigraduation given by
    \[
    (S^{-1}R)_n=\left\{\frac{f}{s}\mid f\in \bigcup_{k\in\mathbb{Z}} R_k,\ {\rm deg}(f) - {\rm deg}(s)=n \right\}.
    \]
    \begin{proof}
        It is clear that $R$ is an $\mathsf{LSG}$ $R$-module, so $S^{-1}R$ is an $\mathsf{SG}$ $R$-module with the semigraduation above, so $S^{-1}R = \bigoplus_{k\in\mathbb{Z}}(S^{-1}R)_k$. It is easy to see that $1=\frac{1}{1}\in(S^{-1}R)_0$. We only have to show that  $(S^{-1}R)_n(S^{-1}R)_m\subseteq\bigoplus_{k\leq n+m}(S^{-1}R)_k$.

        \medskip

        Let $\frac{r_1}{s_1}\in(S^{-1}R)_n$ and $\frac{r_2}{s_2}\in(S^{-1}R)_m$. There exist elements $u\in R$ and $v\in S$ such that $vr_1=us_2$, which implies that $\frac{r_1}{s_1}\frac{r_2}{s_2}=\frac{ur_2}{vs_1}$. Again, since $s_2,v\in R''$ and $r_1$ is homogeneous, we can take $u$ as an homogeneous element. Hence, $ur_2\in\bigoplus_{k\leq {\rm deg}(u) + {\rm deg}(r_2)}R_k$, and so
        $\frac{ur_2}{vs_1}\in\bigoplus_{k\leq {\rm deg}(u) + {\rm deg}(r_2) - {\rm deg}(v) - {\rm deg}(s_1)}(S_{-1}R)_k$, i.e., $\frac{ur_2}{vs_1}\in \bigoplus_{k\leq n+m}(S^{-1}R)_k$.
    \end{proof}
\end{proposition}

\begin{proposition}
Let $R$ be an $\mathsf{SG}$ ring, $S$ a good left Ore set and $M$ an $\mathsf{LSG}$ $R$-module. Then $S^{-1}M$ is an $\mathsf{SG}$ $S^{-1}R$-module.
\end{proposition}
\begin{proof}
We know that $S^{-1}M$ is an $S^{-1}R$-module and an $\mathsf{SG}$ $R$-module, which implies the direct sum $S^{-1}M = \bigoplus_{k\in\mathbb{Z}}(S^{-1}M)_k$. In this way, we just have to prove that $(S^{-1}R)_n(S^{-1}M)_m\subseteq\bigoplus_{k\leq n+m}(S^{-1}M)_k$. Consider elements $\frac{r}{s_1}\in(S^{-1}R)_n$ and $\frac{a}{s_2}\in(S^{-1}M)_m$. There exist elements $u\in R$ and $v\in S$ such that $vr = us_2$, which implies that $\frac{r}{s_1}\frac{a}{s_2} = \frac{ua}{vs_1}$. Again, since $s_2,v\in R''$ and $r$ is homogeneous, we can take $u$ as an homogeneous element. Hence, $ua\in\bigoplus_{k\leq {\rm deg}(u) + {\rm deg}(a)}R_k$, and so $\frac{ua}{vs_1}\in\bigoplus_{k\leq {\rm deg}(u) + {\rm deg}(a) - {\rm deg}(v) - {\rm deg}(s_1)}(S_{-1}R)_k$, i.e., $\frac{ua}{vs_1}\in \bigoplus_{k\leq n+m}(S^{-1}R)_k$.
\end{proof}

%Since we want to perform localizations and limits, we must consider modules bounded by some value of $\beta$. However, as it is clear, this $\beta$ should be such that if $M$ is $\beta-\mathsf{SG}$, then $S^{-1}M$ is also $\beta-\mathsf{SG}$. With this aim, notice that at the end of the proof of Proposition \ref{localization-module} we should have the relationship
%\[
%r\frac{m}{s}\in\bigoplus_{k=\beta(a,b-c)}^{a+b} (S^{-1}M)_k,
%\]
%where ${\rm deg}(r) = a, {\rm deg}(m) = b$ and ${\rm deg}(s) = c$. In this case, $r\frac{m}{s}=\frac{r'm}{s'}\in\bigoplus_{k=\beta(d,b)-e}^{c+d-e}(S^{-1}M)_k$, with ${\rm deg}(r') = d$, ${\rm deg}(s') = e$ and $r's=s'r$, whence $d+c=e+a$. Hence, for the condition to be met it is sufficient that $\beta(a,b-c)\leq \beta(d,b)-e = \beta(a+e-c,b)-e$, for every elements $a,b,c,e\in \mathbb{Z}$. Equivalently, the following inequality has to be satisfied 
%\[
%\beta(a+e-c,b)\geq \beta(a,b-c)+e. 
%\]
%In particular, if $a=c=0$ then $\beta(e,b)\geq\beta(0,b)+e$, and if $a=0$ and $b=e=c$, it follows that $\beta(0,b)\geq \beta(0,0)+b$. Both inequalities guarantee that $\beta(e,b)\geq \beta(0,0)+e+b=\beta_k(e,b)$, where $k=-\beta(0,0)$.

\subsection{Category of semi-graded rings}
    We define the {\em category} $\mathsf{SGR}$ {\em of semi-graded rings} whose objects are the semi-graded rings and morphisms are the homogeneous ring homomorphisms. For a semi-graded ring $R$, $\mathsf{SGR}-R$ will denote the {\em category of semi-graded modules over} $R$ where the morphisms are the homogeneous $R-$homomorphisms. It is straightforward to see that $\mathsf{SGR}-R$ is preadditive, and that the zero object of the category is the trivial module.

\medskip

Let $f:M\rightarrow N$ be a morphism in $\mathsf{SGR}-R$. Since ${\rm Ker}(f)$ and ${\rm Im}(f)$ are semi-graded submodules, it follows that $N/{\rm Im}(f)$ is a semi-graded module. This fact guarantees that the category $\mathsf{SGR}-R$ has kernels and cokernels. If $f$ is a monomorphism of $\mathsf{SGR}-R$, then $f$ is the kernel of the canonical homomorphism $j:N\rightarrow N/{\rm Im}(f)$. If $f$ is an epimorphism, then $f$ is the cokernel of the inclusion $i:{\rm Ker}(f)\rightarrow M$. In this way, the category $\mathsf{SGR}-R$ is normal and conormal.

\medskip

If $\{M_i\}_{i\in I}$ is a family of objects of $\mathsf{SGR}-R$, then their direct sum $\bigoplus_{i\in I} M_i$ is a semi-graded ring with semi-graduation given by
\[
\biggl(\bigoplus_{i\in I} M_i\biggr)_p := \bigoplus_{i\in I}(M_i)_p, \ p\in\mathbb{Z}.
\]
It is easy to see that this object with the natural inclusions coincides with the coproduct of the familiy of objects $\{M_i\}_{i\in I}$ in $\mathsf{SGR}-R$. Therefore, $\mathsf{SGR}-R$ is an Abelian category.

\medskip 

We define $\mathsf{LSG}-R$ as the full subcategory of $\mathsf{SGR}-R$ whose objects are the  $\mathsf{LSG}$ modules. This subcategory is closed for subobjects, quotients and coproducts, so it is Abelian.

\section{Schematicness of semi-graded rings}\label{schematicness}

Following Van Oystaeyen and Willaert's ideas developed in \cite{VanOystaeyenWillaert1995}, in this section we define the notion of {\em schematicness} in the setting of semi-graded rings. For a positively $\mathsf{SG}$ ring $R$, we define $R_+=\bigoplus_{k\ge 1} R_k$ and we say that a left Ore set $S$ is {\em non-trivial} if $S\cap R_+\neq\emptyset$.

\begin{definition}\label{def.schematic}
    Let $R$ be a positively $\mathsf{SG}$ left Noetherian ring. $R$ is called ({\em left}) {\em schematic} if there is a finite set $I$ of non-trivial good left Ore sets of $R$ such that for each $(x_S)_{S\in I}\in \prod_{S\in I}S$, there exist $t,m\in \mathbb{N}$ such that $(R_{\ge t})^m\subseteq\sum_{S\in I}Rx_s$.
\end{definition}

The following result illustrates some characterizations of being schematic (c.f. Definition \ref{schematicgradedsetting}).

\begin{proposition}
    Let $R$ be a positively $\mathsf{SG}$ left Noetherian algebra and $S_1,\dots, S_n$ a finite set of non-trivial good left Ore sets of $R$. The following conditions are equivalent:
    \begin{enumerate}
        \item [\rm (1)] For each $(x_1,\dots,x_n)\in \prod_{i=1}^n S_i$, there exist elements $t, m\in \mathbb{N}$ such that $(R_{\ge t})^m\subseteq\sum_{S\in I}Rx_s$.
        \item [\rm (2)] Let $I\vartriangleleft_l R$. If $I$ has no trivial intersection with every $S_i$, then $I$ contains a power of $R_{\ge t}$ for some $t\in\mathbb{N}$.
        \item [\rm (3)] $\bigcap_{i=1}^n\mathscr{L}(S_i)=\mathscr{L}(\kappa_+)$.
    \end{enumerate}
    \begin{proof}
        The equivalence $(1) \Leftrightarrow (2)$ and the implication $(3) \Rightarrow (1)$ are straightforward.

        $(1) \Rightarrow (3)$ Let $I\in \bigcap_{i=1}^n\mathscr{L}(S_i)$. There exist elements $x_1,\dots, x_n$ such that $x_i\in I\cap S_i$, for every $i$. Thus $\sum_{i=1}^n Rx_i\subseteq I$, and there exist $t,m$ with $(R_{\ge t})^m\subseteq I$, which shows that $I\in\mathscr{L}(\kappa_+)$.

        Now, let $I\in \mathscr{L}(\kappa_+)$. There exist $t,m$ such that $(R_{\ge t})^m\subseteq I$. By using that $S_i\cap R_+\neq \emptyset$, there exist elements $s_i\in S_i$ such that ${\rm deg}(s_i)\ge 1$, for all $i$. Hence $s_i^t\in R_{\ge t}$, $s_i^{tm}\in (R_{\ge t})^m\subseteq I$, and therefore $I\cap S_i\neq\emptyset$. This shows that $I\in\bigcap_{i=1}^n\mathscr{L}(S_i)$.
    \end{proof}
\end{proposition}

If $R$ is schematic by considering the {\em good} left Ore sets $S_i$, then $\bigcap_{i=1}^n\kappa_{S_i}(M) = \kappa_+(M)$, for every $R$-module $M$. If $M$ is an $\mathsf{LSG}-$module, then for each $i = 1,\dots,n$ we have that $\kappa_{S_i}(M)$ is an $\mathsf{SG}$ submodule, and so $\kappa_+(M)$ is also an $\mathsf{SG}$ submodule. These facts imply that $M/\kappa_{+}(M)$ is an $\mathsf{SG}$ $R$-module, and so a submodule of $Q_{\kappa_+}(M)$. The idea is to show that $Q_{\kappa_+}(M)$ is semi-graded. For the remainder of the section we will take $\mathscr{L} := \mathscr{L}(\kappa_+)$.

\medskip

Let us start by taking an $\mathsf{LSG}$ $R$-module $M$ such that $\kappa_+(M)=0$. It is clear that $Q_{\kappa_+}(M) = M_{\mathscr{L}}$ and $\varphi_M(M)\cong M$. Thus, $\varphi_M(M)$ is a submodule of $M_\mathscr{L}$ which is an $\mathsf{SG}$ $R$-module where $\varphi_M(m)$ is homogeneous of degree $k$ if and only if $m$ is homogeneous of degree $k$. If we want $M_{\mathscr{L}}$ to be an $\mathsf{LSG}$ $R$-module, it must be satisfied that if $\xi$ is homogeneous of degree $k$, then for every $s\in R'$, the element $s\xi\in (\varphi_M(M))_{{\rm deg}(s)+k}$. Now, since there exists $I\in\mathscr{L}$ with $I\xi\subseteq \varphi_M(M)$, the following definition makes sense.

\begin{definition}
    Let $\xi\in M_{\mathscr{L}}$. We say that the element $\xi$ is {\em homogeneous of degree} $k$ if there exists $I\in\mathscr{L}$ such that $I\xi\subseteq\varphi_M(M)$, and for every element $s\in I\cap R'$, $s\xi\in (\varphi_M(M))_{{\rm deg}(s)+k}$.
\end{definition}

Notice that if the condition above is satisfied for $I$, then it also holds for every $J\subseteq I$. The following lemma shows that this condition is true for ideals containing $I$.

\begin{remark}
Since the good Ore sets $S_i$ are non-trivial, there exist elements $s'_i\in S_i\cap R_+$ for $i=1,\dotsc,n$, whence $\alpha_i = {\rm deg}(s'_i)>0$. If we define $m := {\rm lcm}\{\alpha_i\}_{1\le i\le n}$ and $s_i'':=(s'_i)^{m/\alpha_i}$, then we obtain $s_i''\in S_i\cap R_+$, and all of them have the same degree. Now, if we consider an element $I\in \mathscr{L}$, there exist $t,n\in \mathbb{Z}$ such that $(R_{\ge t})^n\subseteq I$. Thus, $s_i=(s_i'')^{tn}\in I\cap S_i$, which implies that $\sum_{i=1}^n Rs_i\subseteq I$. In this way, for each $I\in \mathscr{L}$ there exist elements $s_i\in S_i$, all with the same positive degree, satisfying the relation $\sum_{i=1}^n Rs_i\subseteq I$.
\end{remark}

\begin{lemma}\label{lemmasameproperty}
    Let $I, J\in\mathscr{L}$ be ideals such that $I\subseteq J$ and $I\xi,J\xi\subseteq \varphi_M(M)$. If for every $s\in I\cap R'$ the element $s\xi\in (\varphi_M(M))_{{\rm deg}(s)+k}$, then the same property holds for each  $s\in J\cap R'$. 
\end{lemma}
\begin{proof}
Let $s\in J\cap R'$. Then $s\xi\in\varphi_M(M)$ and so there exist homogeneous elements $\xi_j,\ j=l_1,\dots,l_r$ of $\varphi_M(M)$ with $\xi_j\in(\varphi_M(M))_j$ and such that $s\xi = \sum \xi_j$. As we said before, if the property holds for $I$ then it is true for any ideal contained in $I$, so we can take $I=(R_{\ge t})^m$, for some $t,m\in\mathbb{N}$. From above, there exist $s_i\in S_i$ for $i=1,\dots,n$ such that $\sum Rs_i\subseteq I$ and ${\rm deg}(s_i) = \beta$ for each $1\leq i\leq n$. In particular, every element $s_i\in I$ whence $s_is\in I$ (recall that $I$ is a two-sided ideal). By assumption, $s_is\xi\in(\varphi_M(M))_{{\rm deg}(s)+\beta+k}$ for each $i$.

On the other hand, if we consider the expression $s_is\xi=\sum s_i\xi_j$ in terms of homogeneous elements of $\varphi_M(M)$, then for each $j\neq k + {\rm deg}(s)$ the equality $s_i\xi_j=0$ holds. Since this is true for every $i$, it follows that $\xi_j\in\bigcap_{i=1}^n\kappa_{S_i}(\varphi_M(M))=\kappa_+(\varphi_M(M))=0$ (recall that $\kappa_+(M)=0$). Therefore, $s\xi=\xi_{{\rm deg}(s)+k}$.
\end{proof}

From Lemma \ref{lemmasameproperty}, it is sufficient to guarantee the property by considering any ideal $I$ such that $I\xi\subseteq \varphi_M(M)$. Our purpose is to give a more simple method to verify that the element $\xi$ is homogeneous. Let $\xi=[I,f]$. Since $I\xi\subseteq \varphi_M(M)$, the element $\xi$ is homogeneous of degree $k$ if and only if for each $s\in I\cap R'$ the element $s\xi=[R,\beta(f(s))]=\varphi_M(f(s))\in(\varphi_M(M))_{{\rm deg}(s)+k}$, or equivalently, for all $s\in I\cap R'$, the element $f(s)\in M_{{\rm deg}(s)+k}$.

\medskip

For a morphism $f:I\rightarrow M$, we will say that $f$ is {\em homogeneous of degree} $k$ if for each $s\in I\cap R'$, the element $f(s)$ is homogeneous of degree ${\rm deg}(s)+k$. Hence, $[I,f]$ is homogeneous of degree $k$ (in $M_\mathscr{L}$) if and only if $f$ is homogeneous of degree $k$. Let $(M_{\mathscr{L}})_k$ be the family of homogeneous elements of degree $k$. It is clear that $(M_{\mathscr{L}})_k$ is a subgroup and $\varphi_M(M_k) \subseteq (M_{\mathscr{L}})_k$. 

\begin{remark}
We will say that the morphism $f:I\rightarrow M$ is {\em strongly homogeneous of degree} $k$ if for every homogeneous element $s\in I$, the element $f(s)$ is homogeneous of degree ${\rm deg}(s)+k$. It is clear that in the setting of graded rings, the notions of homogeneous morphism and strongly homogeneous morphism coincide.

\medskip

On the other hand, $[I,f]$ it will be called {\em strongly homogeneous of degree} $k$ if some of its representative elements is strongly homogeneous of degree $k$. Let $(\overline{M_{\mathscr{L}}})_k$ be the family of strongly homogeneous elements of degree $k$. It is straightforward to see that $\overline{R_{\mathscr{L}}} = \bigoplus (\overline{R_{\mathscr{L}}})_k$ is a graded ring and  $\overline{M_{\mathscr{L}}} = \bigoplus (\overline{M_{\mathscr{L}}})_k$ is an  $\overline{R_{\mathscr{L}}}$-graded module. Note also that if $s\in R''$ then $\varphi_R(s)\in \overline{R_{\mathscr{L}}}$; in particular $\overline{R_{\mathscr{L}}}$, is an extension of the graded ring $\bigoplus R''_k$. As it is clear, $\overline{R_{\mathscr{L}}}$ is an $R$-submodule of $R_{\mathscr{L}}$ if and only if $R$ is graded. This last remark shows that in the setting of non-graded rings is not appropriate to consider strongly homogeneous morphisms.
\end{remark}
        
\begin{proposition}\label{directsum}
The sum $\sum (M_{\mathscr{L}})_k$ is direct.
\end{proposition}
\begin{proof}
Let $[I_i,f_i]\in (M_{\mathscr{L}})_{k_i}$ for $i = 1,\dots,m$ with $k_i\neq k_j$ if $i\neq j$. Notice that if $\sum [I_i,f_i]=0$, then there exists $J\subseteq \bigcap I_i$, $J\in\mathscr{L}$, such that $(\sum f_i)|_J=\sum f_i|_J=0$. We can take $J=\sum Rs_j$ for some $s_j\in S_j$. Let $s\in J\cap R'$ with ${\rm deg}(s) = l$. Then $0=(\sum f_i)(s)=\sum f_i(s)$, and since  $f_i(s)$ is homogeneous of degree $l+k_i$ and all elements $k_i$ are different, then we have a sum of homogeneous elements of different degrees equal to zero, whence $f_i(s)=0$, for each $i$. In particular, $f_i(s_j)=0$, for all $i, j$. Therefore, $f_i(x)=0$ for all $x\in J$, and so $[J,f_i|J]=[I_i,f_i]=0$.
\end{proof}

Let $[I,f]\in M_\mathscr{L}$ with $I=\sum_{i=1}^n Rs_i$, for some elements $s_i\in S_i\cap R''_{k}$. Since there are finitely $s_i$'s, we can consider that the homogeneous decompositions of the elements $f(s_i)$ have the same length, say $f(s_i)=\sum_{t=\alpha}^\beta (f(s_i))_{t+k}$, where $(f(s_i))_j$ is the $j$-th homogeneous component of $f(s_i)$. By taking  $f_t(s_i)=(f(s_i))_{t+k}$, we have  $f(s_i)=\sum_{t=\alpha}^\beta f_t(s_i)$. For elements $t=\alpha,\dots,\beta$, we define the maps $f_t: I\rightarrow M$ in the natural way as $f_t(\sum a_is_i)=\sum a_if_t(s_i)$. However, we have to prove that these maps are well defined. This is the content of the following proposition.

\begin{proposition}\label{MHTf_t}
$f_t$ is well defined for every element $t=\alpha,\dots,\beta$.
 \end{proposition}
\begin{proof}
We divide the proof in three parts.
\begin{itemize}
    \item Suppose that $0=\sum a_is_i$, with $a_i\in R_{k_1}$ for every $i$ (recall that $s_i\in R''_{k}$). Fix $i$. Since $s_j\in R''$ for each $1\leq j\leq n$, there exist elements $u_j\in R'$ and $v_j\in S_i$ such that $u_js_i=v_js_j$. In particular, ${\rm deg}(u_j) = {\rm deg}(v_j)$, and since $u_j,v_j\in R'$, $u_jf(s_i)=v_jf(s_j)$, and $M$ is $\mathsf{LSG}$, if we compare the homogeneous components of the same degree, then we obtain that $u_jf_t(s_i) = v_jf_t(s_j)$, for each $\alpha\leq t\leq \beta$.  

Now, by using that $v_1\in S_i$ and $a_1\in R$, there exist elements $b_1\in R$ and $c_1\in S_i$ such that $b_1v_1 = c_1a_1$. Repeating this argument with the elements $v_2$ and $c_1a_2$, we find that $b_2\in R$ and $c_2\in S_i$ satisfy the equality $b_2v_2 = c_2c_1a_2$. Continuing in this way, for every $1\leq j\leq n$ we will find elements $b_j\in R$ and $c_j\in S_i$ such that $b_jv_j = \prod_{i=1}^jc_ia_j$ (notice that the elements $b_i$'s can be taken homogeneous). If we define $c := \prod_{i=1}^nc_i\in S_i$ and $d_j = \prod_{i=j+1}^nc_ib_j$, then we have $d_jv_j = ca_j$, for every $1\leq j\leq n$. Hence $0 = c\sum_{j=1}^n a_js_j = \sum_{j=1}^n d_jv_js_j = \sum_{j=1}^nd_ju_js_i = rs_i$, where $r = \sum_{j=1}^nd_ju_j$. Note that the elements $d_ju_j$ are homogeneous of the same degree, which implies that $r$ is also homogeneous. Since $0 = rs_i$, by the first condition of the noncommutative localization, there exists an element $s\in S_i$ such that $sr = 0$. 

Now, considering the equalities
\[
sc\sum_{j=1}^na_jf_t(s_j) = s\sum_{j=1}^nd_jv_jf_t(s_j) = s\sum_{j=1}^nd_ju_jf_t(s_i) = sr f_t(s_i) = 0,
\]
it follows that  $\sum_{j=1}^na_jf_t(s_j)\in\kappa_{S_i}(M)$. Since this holds for every element $i$, we have that $\sum_{j=1}^na_jf_t(s_j)\in\bigcap\kappa_{S_i}(M) = \kappa_+(M)=0$, whence $\sum_{j=1}^na_jf_t(s_j)=0$.

\item Suppose that $0 = \sum_{i=1}^n a_is_i$ (the elements $a_i$'s are not necessarily homogeneous). Since there are only finitely elements $a_i$'s, we can consider the sum $a_i = \sum_{j=l_1}^{l_2}b_{i,j}$, with $b_{i,j}\in R_j$. In this way, $0 = \sum_{i=1}^n\sum_{j=l_1}^{l_2} b_{i,j}s_i = \sum_{j=l_1}^{l_2}\sum_{i=1}^n b_{i,j}s_i$. Now, using that $\sum_{i=1}^n b_{i,j}s_i\in R_{j+k}$ is the homogeneous component of degree $j+k$, it follows that $0 = \sum_{i=1}^n b_{i,j}s_i$. By the first part above, we can assert that $\sum_{i=1}^nb_{i,j}f_t(s_i)=0$, whence $0 = \sum_{j=l_1}^{l_2}\sum_{i=1}^nb_{i,j}f_t(s_i)=\sum_{i=1}^n\sum_{j=l_1}^{l_2}b_{i,j}f_t(s_i)=\sum_{i=1}^na_if_t(s_i)$, for $\alpha\leq t\leq\beta$.

\item Let $r$ be an element of $\sum Rs_i$. Suppose that we have two expressions for $r$ given by $r=\sum a_is_i=\sum b_is_i$. Then $0 = \sum (a_i-b_i)s_i$. By the second part above, $\sum(a_i-b_i)f_t(s_i) = 0$, and so $\sum a_if_t(s_i) = \sum b_if_t(s_i)$, for $\alpha\leq t\leq\beta$. This means that the expression for $f_t(r)$ does not depend of the decomposition of $r$.
\end{itemize}
\end{proof}

From the proof of Proposition \ref{MHTf_t} it follows that the maps $f_t$'s are $R-$ho\-mo\-mor\-phisms. The next proposition establishes that these are homogeneous of degree $t$.

\begin{proposition}\label{prop.f_t.is.homogeneous}
The map $f_t$ is homogeneous of degree $t$.
\end{proposition}
\begin{proof}
Consider $s\in I\cap R'$ with ${\rm deg}(s) = l$. Let $(f_t(s))_m$ be the homogeneous component of degree $m$ in the expression of $f_t(s)$. For a fixed $i$, there exist elements $v_i\in S_i$ and $u_i\in R'$ such that $u_is_i = v_is$, which implies that $v_if_t(s)=u_if_t(s_i)$. Since $f_t(s_i)\in M_{t+k}$, $u_i,v_i\in R''$ and $M$ is $\mathsf{LSG}$, when we compare the homogeneous components of these elements, we have that if $m\neq t+l$ then $v_i(f_t(s))_m=0$, whence $(f_t(s))_m\in \kappa_{S_i}(\varphi_M(M))$. Since this fact holds for every $i$, it follows that $(f_t(s))_m\in \bigcap \kappa_{S_i}(\varphi_M(M)) = \kappa_+(\varphi_M(M))=0$. Therefore, $f_t(s) = (f_t(s))_{t+l}$, which asserts that $f_t$ is homogeneous of degree $t$.
 \end{proof}

Propositions \ref{directsum}, \ref{MHTf_t} and \ref{prop.f_t.is.homogeneous} imply the following important result.

\begin{proposition}\label{preliminartheorem}
    If $M$ is an $\mathsf{LSG}$ $R$-module with $\kappa_+(M)=0$, then $Q_{\kappa_+}(M) = M_{\mathscr{L}}$ is an $\mathsf{LSG}$ $R$-module with semigraduation given by 
    \[ 
    M_{\mathscr{L}} = \bigoplus_k (M_{\mathscr{L}})_k.
    \]
\end{proposition}

\begin{theorem}
If $M$ is an $\mathsf{LSG}$ $R-$module, then $Q_{\kappa_+}(M)$ is an $\mathsf{LSG}$ $R-$module.
\end{theorem}
\begin{proof}
   From Proposition \ref{preliminartheorem} it follows that $\kappa_+\left(M/\kappa_+(M)\right) = 0$ and $Q_{\kappa_+}(M) = Q_{\kappa_+}(M/\kappa_+(M))$.
\end{proof}

\section{Serre-Artin-Zhang-Zerevkin theorem}\label{section-serre-theorem}

In this section, we prove the Serre-Artin-Zhang-Zerevkin theorem for semi-graded rings (Theorem \ref{Serre-theo}) using a different approach than the one presented by Lezama \cite{Lezama2021, LezamaLatorre2017}.

\medskip

Briefly, this theorem was partially formulated by Lezama and Latorre \cite[Theorem 6.12]{LezamaLatorre2017} where it was assumed that the semi-graded left Noetherian ring is a domain. Nevertheless, as is well-known, the Serre-Artin-Zhang-Verevkin theorem for finitely graded algebras does not include this restriction, so that this assumption was eliminated by Lezama \cite[Theorem 1.24]{Lezama2021} (see also \cite[Section 18.4, Theorem 18.5.13]{libropbw}). More exactly, he proved the theorem for an $\mathsf{SG}$ ring $R = \bigoplus_{n\ge 0} R_n$ satisfying the following conditions:

\begin{enumerate}
    \item [\rm (C1)] $R$ is left Noetherian;
    \item [\rm (C2)] $R_0$ is left Noetherian;
    \item [\rm (C3)] for every $n$, $R_n$ is a finitely generated left $R_0$-module;
    \item [\rm (C4)] $R_0 \subset Z(R)$.
\end{enumerate}
Notice that condition (C4) implies that $R_0$ is a commutative Noetherian ring.

\medskip

Universal enveloping algebras of finite-dimensional Lie algebras, some quantum algebras with three generators, and some examples of 3-dimensional skew polynomial algebras \cite{BellSmith1990, Redman1999, ReyesSarmiento2022} illustrate the Serre-Artin-Zhang-Verevkin theorem \cite[Example 1.26]{Lezama2021} and \cite[Example 18.5.15]{libropbw}.

\medskip

We start with the following preliminary result.

\begin{lemma}\label{preliminaryresult}
Let $R$ be a positively $\mathsf{SG}$ left Noetherian ring and $S$ a non-trivial left Ore set of $R$. Then $\mathscr{L}(\kappa_+)\subseteq \mathscr{L}(S)$.
\end{lemma}
\begin{proof}
Let $I\in\mathscr{L}(\kappa_+)$. There exist elements $t,n\in \mathbb{N}$ such that $R_{\ge t}^n\subseteq I$. Since $S$ is non-trivial, there exists $s\in S$ with $\deg(s)\ge 1$, whence $s^{tn}\in R_{\ge t}^n$. This fact shows that $S\cap I\neq\emptyset$.
\end{proof}

Lemma \ref{preliminaryresult} says that if $M$ is an $R$-module and $S$ is a non-trivial left Ore set of $R$, then $\kappa_+(M)\subseteq \kappa_S(M)$.

\begin{lemma}\label{iso.local}
    Let $R$ be a positively $\mathsf{SG}$ left Noetherian ring and $S$ a non-trivial good left Ore set. If $M$ is an $\mathsf{LSG}$ $R$-module, then $S^{-1}(M)\cong S^{-1}(Q_{\kappa_+}(M))$.
\end{lemma}
    \begin{proof}
        Let 
        \begin{align*}
        f\colon S^{-1}M & \longrightarrow S^{-1}(M/\kappa_+(M)) \\
        \frac{m}{s} & \longmapsto \frac{\overline{m}}{s}.
    \end{align*}
    It is clear that $f$ is surjective. Let $\frac{m}{s}\in {\rm Ker}(f)$. Then $\frac{\overline{m}}{s}=0$, and so there exists $s'\in S$ such that $s'\overline{m}=0$, i.e., $s'm\in\kappa_+(M)\subseteq \kappa_S(M)$. Hence, there exists $s''\in S$ with $s''s'm=0$, and since $s''s'\in S$, it follows that $\frac{m}{s}=0$. Therefore, $S^{-1}(M)\cong S^{-1}(M/\kappa_+(M))$.

    Now, let 
    \begin{align*}
    g\colon S^{-1}(M/\kappa_+(M)) & \longrightarrow S^{-1}(Q_{\kappa_+}(M)) \\        \frac{\overline{a}}{s} & \longmapsto \frac{h(\overline{a})}{s},
    \end{align*}   
    where $h$ is the isomorphism between $M/\kappa_+(M)$ and $\varphi_M(M)$. Since $h$ is injective, so $g$ also is. Let $\frac{\xi}{s}\in S^{-1}(Q_+(M))$. Then there exist elements $t,n\in\mathbb{N}$ such that $(R_{\ge t}^n)\xi\subseteq \varphi_M(M)$. Since $S$ is non-trivial, repeating the argument above in the proof of Lemma \ref{preliminaryresult} we can assert that there exists $s'\in S$ such that $s''=(s')^{tn}\in R_{\ge t}^n$. In this way, $s''\xi\in\varphi_M(M)$, and so there exist $m\in M$ such that $s''\xi = \varphi_M(m)$, whence $g\left(\frac{\overline{m}}{s''s}\right) = \frac{h(\overline{m})}{s''s}=\frac{\varphi_M(m)}{s''s} = \frac{s''\xi}{s''s} = \frac{\xi}{s}$. We conclude that $S^{-1}(Q_+(M))\cong S^{-1}(M/\kappa_+(M)) \cong S^{-1}(M)$.
    \end{proof}

For the rest of this section, $R$ denotes a schematic ring (recall that by Definition \ref{def.schematic} $R$ is left Noetherian). Consider the full subcategory $(R,\kappa_+)-\mathsf{LSG}$ of $\mathsf{LSG}-R$ whose objects are the $\kappa_+$-closed modules. If $M$ is an $R$-module $\kappa_+$-closed and $N$ is a submodule of $M$, then $N$ is $\kappa_+$-closed if and only if $M/N$ is $\kappa_+$-torsion-free \cite[Proposition 4.2, Chapter IX]{Stenstrom1975}. Hence, it is clear that the intersection of $\kappa_+$-closed modules is $\kappa_+$-closed. This fact allows us to consider the submodule $\kappa_+$-closed generated by a subset of $M$. If we define
\begin{equation*}
    N^c=\{x\in M\mid (N:x)\in\mathscr{L}(\kappa_+)\},
\end{equation*}
then it is clear that $N^c$ is the submodule $\kappa_+$-closed generated by $N$, and in fact, $N^c=M$ if and only if  $M/N$ is $\kappa_+$-torsion.

\medskip

Notice that in the category $(R,\kappa_+)-\mathsf{LSG}$ the subobjects are the submodules $\mathsf{LSG}$-$\kappa_+$-closed, that are closed under arbitrary intersections. The submodule $\mathsf{LSG}$-$\kappa_+$-closed generated by $X\subseteq M$ will be denoted as $\langle X\rangle^{\mathsf{SG}-\kappa}$. We will say that $M$ is $\mathsf{LSG}$-$\kappa_+$-{\em finitely generated} if there exists a finite set $X\subseteq M$ with $\langle X\rangle^{\mathsf{SG}-\kappa} = M$. Let ${\rm Proj}(R)$ be the $\mathsf{LSG}$-$\kappa_+$ category of finitely generated $R$-modules.

\begin{proposition}
If $N$ is an $\mathsf{SG}$ submodule of $M$, then $N^c$ also is.
\begin{proof}
Let $m = m_1 + \dotsb + m_k\in N^c$ with $m_i\in M_{l_i}$. There exists $I\in\mathscr{L}(\kappa_+)$ such that $I\subseteq (N:m)$. Since $R$ is schematic by the good left Ore set $S_i$, $i=1,\dots,n$, say, then there exist elements $s_i\in S_i$ with $\sum Rs_i\subseteq I$, whence $s_im\in N$ for all $i$.  Since $N$ is $\mathsf{SG}$ and $s_i\in R''$, then $s_im_j\in N$, for each $i,j$. Thus, $\sum_{i = 1}^n Rs_i\subseteq (N:m_j)$, which shows that $m_j\in N^c$.
\end{proof}
\end{proposition}

From these facts we have the equality $\langle X\rangle^{ \mathsf{SG}-\kappa} = \left(\langle X\rangle^{\mathsf{SG}}\right)^{c}$ for each $X\subseteq M$. In this way, $M$ is $\mathsf{LSG}$-$\kappa_+$-finitely generated if and only if there exist a finite set $X\subseteq M$ such that $\left(\langle X\rangle^{\mathsf{SG}}\right)^{c}=M$, or equivalently, $M/M_1$ is $\kappa_+$-torsion with $M_1=\langle X\rangle^{\mathsf{SG}}$.

\medskip

Next, we define the notion of {\em noncommutative site}.

\begin{definition}
Let $\mathscr{O}$ be the set of non-trivial good left Ore sets of $R$ and $\mathscr{W}$ the free monoid on $\mathscr{O}$. We define the category $\mathscr{\underline{W}}$ as follows: the objects of $\mathscr{\underline{W}}$ are the elements of $\mathscr{W}$, while for two words $W$ and $W'$ we define the morphisms of $\mathscr{\underline{W}}$, denoted by ${\rm Hom}(W',W)$, as a singleton $\{W'\rightarrow W\}$ if there exists an increasing injection from the letters of $W$ to the letters of $W'$, i.e., $W = S_1,\dotsc S_n$ and $W'=V_0S_1V_1S_2V_2\dots S_nV_n$ for some letters $S_i$ and some (possibly empty) words $V_i$. In other case, ${\rm Hom}(W',W)$ is defined to be empty.
\end{definition}

It is easy to see that $\mathscr{\underline{W}}$ is a thin category. We denote the empty word as $1$, which is the final object of the category.

\medskip

If $W = S_1\dotsc S_n\in\mathscr{W}\ \backslash\ \{1\}$ and $M$ is an $\mathsf{LSG}$ $R$-module, we define 
\[
Q_W(M) = S_n^{-1}R\otimes_R\dots\otimes_R S_1^{-1}R\otimes_R M.
\] 
Lemma \ref{iso.local} asserts that if $W\neq 1$, then $Q_W(M)\cong Q_W(Q_{\kappa_+}(M))$.
 
\medskip

If $W = S_1\dots S_n\in\mathscr{W}\ \backslash\ \{1\}$, we say that $w \in W$ if $w= s_1\dots s_n$ with $s_i\in S_i$. We associate a set of left ideals to $W$, namely 
\[
\mathscr{L}(W) = \{I\vartriangleleft_l R\mid\ {\rm there\ exists}\ w\in W\ {\rm  such\ that}\ w\in I\}.
\]
We define $\mathscr{L}(1) = \mathscr{L}(\kappa_+)$.

\begin{lemma}\label{wordsproof}
    Let $W\in \mathscr{W}\ \backslash\ \{1\}$ and $w, w'\in W$. Then there exists $w''\in W$ such that $w''=aw$ and $w''=bw'$, for some elements $a, b\in R$.
\begin{proof}
We prove the assertion by induction on the length of elements of $W$. If $W = S_1$, then by the Ore's condition there exist elements $a\in R$ and $b\in S_1$ such that $aw= bw'\in S_1$.

Suppose that the assertion holds for every element of length $k$. Let $W = S_1\dotsc S_{k+1},\ \Tilde{W} = S_2\dotsc S_{k+1},\ w = s_1\dots s_{k+1},\ w' = s'_1\dots s'_{k+1}\in W$, $x=s_2\dots s_{k+1}$, and $x'=s'_2\dots s'_{k+1}$. By the inductive step, there exist elements $a, b\in R$ such that $ax = bx'\in\Tilde{W}$. Since $S_1$ is a left Ore set, then there exist $s_1''\in S_1$ and $a_1\in R$ such that $a_1s_1=s_1''a$. Hence, $a_1w = a_1s_1x = s_1''ax = s_1''bx'\in W$. Again, by the Ore's condition, there exist $s_1^*\in S_1$ and $b_1\in R$ such that $b_1s_1'=s_1^*s_1''b$, whence $b_1w'=b_1s_1'x'=s_1^*s_1''bx'=s_1^*a_1w\in W$. 
\end{proof}
\end{lemma}

\begin{remark}\label{factor.comun}
    Lemma \ref{wordsproof} can be extended to a finite collection of words, i.e., if $w_1,\dots,w_i\in W$, then there exist $a_1,\dots, a_n\in R$ such that $a_1w_1 = a_2w_2 = \dotsb = a_n w_n\in W$.
\end{remark}

\begin{lemma}\label{wordsproof(2)}
    Let $W\in \mathscr{W}\backslash\ \{1\}$, $w\in W$ and $a\in R$. There exist elements $w'\in W$ and $b\in R$ with $w'a = bw$.
\end{lemma}
\begin{proof}
We prove by induction on the length of words of $W$. If $W = S_1$, then the assertion is precisely the Ore's condition. 

Suppose that the lemma holds for each element of length $k$. Let $W = S_1\dotsc S_{k+1}$, $\Tilde{W} = S_2\dotsc S_{k+1},\ w = s_1\dotsc s_{k+1}$, and $x = s_2\dots s_{k+1}$. By the inductive step, there exist elements $x'\in \Tilde{W}$ and $b\in R$ such that $x'a = bx$. Since $S_1$ is an Ore set, there exist $s_1'\in S_1$ and $b'\in R$ such that $s_1'b = b's_1$, whence $s_1'x'a = s_1'bx = b's_1x = b'w$.
\end{proof}

Lemmas \ref{wordsproof} and \ref{wordsproof(2)} allow us to conclude that $\mathscr{L}(W)$ is a filter. In the case $W\neq1$, we will call $\kappa_W$ the {\em pre-radical} associated to $\mathscr{L}(W)$. It is straightforward to see that for every $\mathsf{LSG}$ $R$-module $M$, the following equality holds
\[
\kappa_W(M) = \{m\in M\mid\ {\rm there\ exists}\ w\in W\ {\rm such\  that}\ wm = 0\} = {\rm Ker}(M\rightarrow Q_W(M)).
\]

Following \cite[p. 113]{VanOystaeyenWillaert1995}, a {\em global cover} is a finite subset $\{W_i \mid i\in I\}$ of $\mathscr{W}$ such that $\bigcap_{i\in I}\mathscr{L}(W_i)=\mathscr{L}(\kappa_+)$.  For $W\in \mathscr{W}$, ${\rm Cov}(W)$ is defined as the set of all sets of the morphisms of $\underline{\mathscr{W}}$ of the form $\{W_iW\rightarrow W \mid i \in I\}$, where $\{W_i\mid i\in I\}$ is a global cover. It is clear that $\{1\}$ is a global cover that will be called trivial. Notice that the schematic condition guarantees the existence of at least one non-trivial global cover. This collection of coverings is not a Grothendieck topology of $\underline{\mathscr{W}}$, but satisfies similar conditions (Proposition \ref{conditions-grothendieck-topology}) that allow us to talk about sheaves on $\underline{\mathscr{W}}$. For this reason, Van Oystaeyen and Willaert called the category $\underline{\mathscr{W}}$ with this coverings the {\em noncommutative site} (c.f. \cite{VanOystaeyenWillaert1996a}).

\medskip

The proof of the following lemma is analogous to the setting of graded rings \cite[Lemma 1]{VanOystaeyenWillaert1995}. We include it for the completeness of the paper. 
\begin{lemma}\label{lemma-global-cover}
    If $\{W_i\mid i\in I\}$ is a global cover, then for all $V\in \mathscr{W}$,  
    \[
    \bigcap_{i\in I}\mathscr{L}(W_iV)=\mathscr{L}(V).
    \]   
\end{lemma}
\begin{proof}
    If $I\in\mathscr{L}(V)$, there exists $v\in V$ such that $v\in I$. Let $w_i\in W_i$, we have that $w_iv\in W_iV$ and $w_iv\in I$ thus $I\in\mathscr{L}(W_iV)$. From this, $\mathscr{L}(V)\subseteq \bigcap_{i\in I}\mathscr{L}(W_iV)$.

    Let $I\in \bigcap_{i\in I}\mathscr{L}(W_iV)$, for each $i$ there exist $v_i\in V$ and $w_i\in W_i$ such that $w_iv_i\in I$. by Remark \ref{factor.comun} there exist $a_1,\dots, a_n\in R$ and $v\in V$ such that $v=a_iv_i$, for each $i$. By Lemma \ref{wordsproof(2)}, we obtain that there exist $w_i'\in W_i$ and $b_i\in R$ with $w_i'a_i=b_iw_i$. Since $w_i'\in \sum Rw'_i$ and $\{W_i\mid i\in I\}$ is a global cover, there exist elements $n, t\in\mathbb{N}$ such that $R_{\ge t}^n\subseteq \sum Rw'_i$. Multiplying by $v$ we obtain $\left(R_{\ge t}^n\right)v\subseteq I$. If $S$ is the first letter of $V$, there exists $s\in S\cap R_{\ge t}^n$. Finally, $sv\in I$ and $sv\in V$, and so $I\in \mathscr{L}(V)$.  
\end{proof}

\begin{proposition}\label{conditions-grothendieck-topology}
    The category $\underline{\mathscr{W}}$ together with the sets ${\rm Cov}(W)$ for any element $W\in\underline{\mathscr{W}}$ satisfy the following properties:
    \begin{enumerate}
        \item [$G_1:$]$\{W\rightarrow W\}\in{\rm Cov}(W)$,
        \item [$G_2:$] $\{W_i\rightarrow W\mid i\in I\}\in {\rm Cov}(W)$ and $\forall i\in I:\{W_{ij}\rightarrow W_i\mid j\in I_i\}\in {\rm Cov}(W_i)$ then $\{W_{ij}\rightarrow W_i\rightarrow U\mid i\in I, j\in I_i\}\in {\rm Cov}(W)$,
        \item [$G_3:$] if $\{W_iW\rightarrow W\mid i\in I\}\in {\rm Cov}(W)$, $W'\rightarrow W\in \underline{\mathscr{W}}$, and if we define $W_iW\times_W W'=U_iW'$ then $\{W_iW\times_W W'\rightarrow W'\mid i\in I\}\in{\rm Cov}(W')$.
    \end{enumerate}
\end{proposition}
\begin{proof}
    $G_1$ holds since $\{1\}$ is a global cover,  $G_2$ is a direct consequence of Lemma \ref{lemma-global-cover} and $G_3$ is clear.
\end{proof}

Definitions \ref{def:presheaf} and \ref{def:sheaf} introduces the notion of presheaf and sheaf, respectively, in the setting of the category $\underline{\mathscr{W}}$. 

\begin{definition}\label{def:presheaf}
A {\em presheaf} $\mathscr{F}$ on $\underline{\mathscr{W}}$ is a contravariant functor from $\underline{\mathscr{W}}$ to the category $\mathsf{LSG}-R$ such that for all $W\in \underline{\mathscr{W}}\backslash\{1\}$, the sections $\mathscr{F}(W)$ of $\mathscr{F}$ on $W$ is an $\mathsf{SG}$ $S^{-1}R$-module, where $S$ denotes the last letter of $W$ and $\mathscr{F}(1)$ is an $\mathsf{SG}$ $Q_{\kappa_+}(R)$-module.
\end{definition}

Since $\mathscr{F}(1)$ denotes the global sections, we will denote it as $\Gamma_*(\mathscr{F})$. We abbreviate $\mathscr{F}(V\rightarrow W)$ as $\rho_V^W:\mathscr{F}(W)\rightarrow\mathscr{F}(V)$; if $W = 1$, then we will write $\rho_V$ instead of $\rho_V^1$. 

\begin{definition}\label{def:sheaf}
    A presheaf $\mathscr{F}$ on $\underline{\mathscr{W}}$ is a {\em sheaf} if it satisfies the following two properties:
    \begin{enumerate}
        \item [\rm (i)] {\em Separatedness}: for all element $W\in\mathscr{W}$ and each global cover $\{W_i\mid i\in I\}$, if $m\in\mathscr{F}(W)$ satisfies that for every $i\in I$, $\rho_{W_iW}^W(m)=0$ in $\mathscr{F}(W_iW)$, then $m = 0$.
        \item [\rm (ii)] {\em Gluing}: for all $W\in\mathscr{W}$ and each global cover $\{W_i\mid i\in I\}$, given $(m_i)\in\prod_i\mathscr{F}(W_iW)$ satisfying 
        \[\rho_{W_iW_jW}^{W_iW}(m_i)=\rho_{W_iW_jW}^{W_jW}(m_j),\quad {\rm for\ all}\ (i,j)
        \in I\times I,
        \]
        there exists an element $m\in\mathscr{F}(W)$ such that
        \[\rho_{W_iW}^W(m)=m_i,\ {\rm for\ all}\ i\in I. 
        \]
    \end{enumerate}
\end{definition}  

\begin{remark}\label{sheaf.iff.limit}
A presheaf $\mathscr{F}$ is a sheaf if and only if for every word $W$ and each global cover $\{W_i\mid i\in I\}$, $\mathscr{F}(W)$ (with the arrows given by $\mathscr{F}$) is the limit of the diagram
\begin{equation}\label{limitdiagram}
\xymatrix@R=8pt{
    \mathscr{F}(W_iW) \ar[dr] \ar[r] & \mathscr{F}(W_iW_jW)\\
    \mathscr{F}(W_jW) \ar[ur] \ar[r] & \mathscr{F}(W_jW_iW)}
\end{equation}
\end{remark}
\begin{proof}
Suppose that $\mathscr{F}$ is a sheaf. Let $M$ be an $\mathsf{SG}$ $R$-module with morphisms $f_i:M\rightarrow \mathscr{F}(W_iW)$ which are compatibles with the morphisms $\rho_{W_iW_jW}^{W_iW}$ and $\rho_{W_JW_iW}^{W_iW}$. Consider an element $m\in M$. By using this compatibility, we have that the element $(f_i(m))$ of $\prod_i\mathscr{F}(W_iW)$ satisfies the equality
\[
\rho_{W_iW_jW}^{W_iW}f_i(m)=\rho_{W_iW_jW}^{W_jW}f_j(m), \;\;\; {\rm for\ all}\ (i,j)
        \in I\times I.
\] 
In this way, there is a unique element $m'\in \mathscr{F}(W)$ such that $\rho_{W_iW}^W(m') = f_i(m)$, for each $i\in I$. If we define the map $\beta:M\rightarrow\mathscr{F}(W)$ as $\beta(m) = m'$, then it is clear that $\beta$ is a homogeneous $R$-homomorphism and it is the only one that satisfies the equality $\rho_{W_iW}^W\circ\beta = f_i$, for all $i\in I$. Hence, $\mathscr{F}(W)$ is the limit of the diagram (\ref{limitdiagram}).

On the other hand, suppose that $\mathscr{F}(W)$ is the limit of the diagram (\ref{limitdiagram}). Let $m\in \mathscr{F}(W)$ such that $\rho_{W_iW}^W(m) = 0$, for each $i\in I$. This means that $m\in\bigcap{\rm Ker}(\rho_{W_iW}^W)$, and by assumption on $\mathscr{F}(W)$, $\bigcap{\rm Ker}(\rho_{W_iW}^W) = 0$, whence $m=0$. Let 
\[
A := \{(m_i)\in\prod\mathscr{F}(W_iW)\mid \rho_{W_iW_jW}^{W_iW}(m_i)=\rho_{W_iW_jW}^{W_jW}(m_j),\ {\rm for\ all}\ (i,j)\in I\times I\}.
\]
It is clear that $A$ is an $\mathsf{SG}$ $R$-submodule of $\prod\mathscr{F}(W_iW)$, which guarantees the existence of only one homogeneous $R-$homomorphism $\beta: A\rightarrow\mathscr{F}(W)$ such that $\rho_{W_iW}^W\circ\beta = \pi_i$, for all $i\in I$, where $\pi_i$ denotes the usual projection. In this way, it is immediate that $\beta(m_i)$ is the element that satisfies the gluing condition.
\end{proof}

\begin{definition}
Let $M$ be an $\mathsf{LSG}$ $R$-module. We define the presheaf $\widehat{M}$ in the following way: for objects, $\widehat{M}(1) = Q_{\kappa_+}(M)$, and for $W\in\mathscr{W}\ \backslash\ \{1\}$, $\widehat{M}(W) = Q_W(M)$. Now, for morphisms, if $W\neq 1$ then to the map $V\rightarrow W$ we assign it the trivial morphism such that the following diagram commutes
\begin{equation}
    \xymatrix{
    Q_W(M) \ar[r]^{\rho_V^W} & Q_V(M)\\
    M \ar[ur] \ar[u] & }
\end{equation}
while for the morphism $W\rightarrow 1$ we assign the composition map
\[
Q_{\kappa_+}(M)\rightarrow Q_W(Q_{\kappa_+}(M))\rightarrow Q_W(M),
\]
where the first arrow is the natural map, and the second arrow is precisely the isomorphism obtained in Lemma \ref{iso.local}.
\end{definition}

If $W = S_1\dotsc S_n$ and $w\in W$, say $w = s_1\dots s_n$, then the element $\frac{1}{s_n}\otimes\dots\otimes \frac{1}{s_1}\otimes m\in Q_W(M)$ will be noted as $\frac{m}{w}$. In particular, $\frac{m}{1}$ stands for $1\otimes m$ in $Q_S(M)$, for $1\otimes 1\otimes m$ in $Q_{ST}(M)$, and so on, which element is meant depends on the module it belongs to.

\medskip

The proof of the following two lemmas follow the same ideas to those presented in the setting of $\mathbb{N}$-graded rings \cite[Lemmas 2 and 3]{VanOystaeyenWillaert1995}.  

\begin{lemma}\label{lemma.van3}
    Given elements $\frac{m}{w}\in Q_W(M)$ and $a\in R$, there exist $w'\in W,\ b\in R$ such that $w'a = bw$ and $a\frac{m}{w} = \frac{bm}{w'}\in Q_W(M)$.
\end{lemma}
\begin{proof}
Let $W = S_1\dotsc S_n$ and $w = s_1\dotsc s_n$. We consider $a_n = a$ and define $a_i$ recursively. More exactly, for an element $a_i$, the Ore's condition guarantees the existence of elements $s_i'\in S_i$ and $a_{i-1}\in R$ such that $s_i'a_i = a_{i-1}s_i$. Hence, $a\frac{m}{w} = \frac{1}{s'_n}\otimes\dotsc \otimes\frac{1}{s_1'}\otimes a_0m$. If we define $b = a_0$ and $w'=s_1'\dotsc s_n'$, then the assertion follows.
\end{proof}

\begin{lemma}\label{lemma.n/1}
    If $\frac{m}{w} = \frac{n}{1}$ in $Q_W(M)$ for some element $n\in M$, then there exist $\tilde{w}\in W$ and $r\in R$ such that $\tilde{w} = rw$ and $\tilde{w}n = rm$.
\end{lemma}
\begin{proof}
Induction on the length $n$ of the word $w = s_1\dotsc s_n$. The case $n=1$ it is clear. Suppose that the assertion holds for any word of length less than $n$. Let $W' = S_1\dotsc S_{n-1}$ and $w' = s_1\dotsc s_{n-1}$. Then $\frac{1}{s_n}\otimes\frac{m}{w'} = 1\otimes \frac{n}{1}\in S_n^{-1}(Q_{W'})$, so that there exist elements $a, b\in R$ such that $a\frac{m}{w'} = b\frac{n}{1}=\frac{bn}{1}\in Q_{W'}(M)$ and $as_n=b\in S_n$. By Lemma \ref{lemma.van3} there exist elements $w''\in W'$ and $c\in R$ with $w''a = cw'$ and $\frac{cm}{w''} = a\frac{m}{w'} = \frac{bn}{1}$. By hypothesis, there exist $w'''\in W'$ and $d\in R$ such that $w'''=dw''$ and $w'''bn=dcm$, so that if we consider $\tilde{w} = w'''b$ and $x = dc$ the assertion follows.
\end{proof}

Let $\{W_i\mid i\in I\}$ be a global cover. The limit of the diagram
\begin{equation}\label{limitdiagramQ}
    \xymatrix{
    Q_W(Q_{W_i}(M)) \ar[r]\ar[rd] & Q_W(Q_{W_j}(Q_{W_i}(M)))\\  
    Q_W(Q_{W_j}(M)) \ar[ur] \ar[r] & Q_W(Q_{W_i}(Q_{W_j}(M)))  }
\end{equation}
will be denoted by $\Gamma_W(\widehat{M})$. Notice that due to the universal property of the limit, for the family $\{M\rightarrow Q_{W_i}\mid i\in I\}$ there is a unique morphism $\varphi: M\rightarrow\Gamma_1(\widehat{M})$. This morphism is of great importance in the following lemma.

\begin{lemma}\label{lemma.torsion}
    Let $\varphi: M\rightarrow\Gamma_1(\widehat{M})$ the morphism described above. Then ${\rm Coker}(\varphi)$ is $\kappa_+$-torsion. 
\end{lemma}
\begin{proof}
Let $\xi = \bigl(\frac{m_i}{w_i}\bigr)_i\in\Gamma_1(\widehat{M})$ with $w_i\in W_i$ and $\frac{1}{w_i}\otimes 1 \otimes m_i=1 \otimes \frac{1}{w_j}\otimes m_j\in Q_{W_i}(Q_{W_j}(M))$, for all $i,j$. Fix $j$. Then $\frac{1\otimes m_i}{w_i} = \frac{\frac{1}{w_j} \otimes m_j}{1}\in Q_{W_i}$, whence by Lemma \ref{lemma.n/1} there exist elements $w_i'\in W_i$ and $a_i\in R$ such that $w_i' = a_iw$ and $w_i'(\frac{1}{w_j} \otimes m_j) = a_i(1\otimes m_i)$, that is, $w_i'(\frac{m_j}{w_j}) = \frac{a_im_i}{1}\in Q_{W_j}(M)$. By taking $I := \sum_{i\in I} Rw_i'$, it is clear that $I\in \bigcap_i\mathscr{L}(W_i) = \mathscr{L}(\kappa_+)$. In this way, there exist elements $t_j, n_j\in\mathbb{N}$ such that $R_{\ge t_j}^{n_j}\subseteq I$, which shows that $(R_{\ge {t_j}}^{n_j})\frac{m_j}{w_j}$ is contained in the direct image of the map $M\rightarrow Q_{W_j}(M)$, that is, $(R_{\ge {t_j}}^{n_j})\frac{m_j}{w_j}\subseteq {\rm Im}(M\rightarrow Q_{W_j}(M))$. If $n :={\rm max}\{n_i\mid i\in I\}$ and $t: ={\rm max}\{t_i\mid i\in I\}$, then it is straightforward to see that this reasoning is true for every element $j$.

Let $a\in R_{\ge t}^n$. Then $a\xi = (\frac{n_i}{1})_i$, for some elements $n_i\in M$ with $1\otimes 1\otimes n_i =  1\otimes 1\otimes n_j$ in $Q_{W_i}(Q_{W_j}(M))$, for every $i,j$. Fix $i$. Lemma  \ref{lemma.n/1} guarantees that for each $j$, there exist elements $\tilde{w}_j\in W_i$ and $x_j\in R$ such that $\tilde{w}_j = x_j$ and $\tilde{w}_j\frac{n_j}{1} = x_j\frac{n_i}{1}$. Now, by Remark \ref{factor.comun}, we can find elements $w^*_i\in W_i$ with $w^*_i\frac{n_i}{1} = w^*_i\frac{n_j}{1}$, for all $j$. Hence,  $w^*_ia\xi = \varphi(w^*_in_i)$. As above, by defining $J = \sum_{i\in I} Rw_i^*$, there exist elements $t'(a)$ and $n'(a)$ (notice that all elements depend on $a$) such that $\left(R_{\ge t'(a)}^{n'(a)}\right)a\xi\subseteq \varphi (M)$. Since $R$ is a left Noetherian ring,  $R_{\ge t}^n$ is finitely generated, say by the elements $a_1,\dotsc a_r$. By defining $n' = {\rm max}\{n(a_k)\mid 1\le k\le r\}$ and $t' = {\rm max}\{t'(a_k)\mid 1\le k\le r\}$, we have $\left( R_{\ge t+t'}^{n+n'}\right)\xi\subseteq \varphi(M)$, i.e., ${\rm Coker}(\varphi)$ is $\kappa_+$-torsion, which concludes the proof.
\end{proof}

\begin{proposition}\label{prop.tildeM.sheaf}
The presheaf $\widehat{M}$ is a sheaf.
\end{proposition}
\begin{proof}
Fix a global cover $\{W_i\mid i\in I\}$ and let $\varphi:M\rightarrow \Gamma_1(\widehat{M})$ be the map established in Lemma \ref{lemma.torsion}. Let us see $\Gamma_1(\widehat{M})\cong Q_{\kappa_+}(M) = \widehat{M}(1)$. Since for the family $\{Q_{\kappa_+}(M)\rightarrow Q_{W_i}\mid i\in I\}$ the universal property of the limit guarantees the existence of a unique morphism $\phi: Q_{\kappa_+}(M)\rightarrow\Gamma_1(\widehat{M})$, we obtain the following commutative diagram
\begin{equation}\label{lastdiagram}
     \xymatrix{
    & \Gamma_1(\widehat{M}) \ar[r]  & Q_{W_i}(M) \ar[r] & Q_{W_j}(Q_{W_i}(M))\\  
    M \ar[r] \ar[ur]^\varphi  & Q_{\kappa_+}(M) \ar[u]^\phi \ar[ur]^{\phi_i}& &   }
\end{equation}
It is clear that ${\rm Ker}(\phi)\subseteq \bigcap {\rm Ker}(\phi_i) = \bigcap \kappa_{W_i}(Q_{\kappa_+}(M)) = \kappa_+(Q_{\kappa_+}(M))=0$. On the other hand, since ${\rm Im}(\varphi)\subseteq {\rm Im}(\phi)$ and $\Gamma_1(\widehat{M})/{\rm Im}(\varphi)$ is $\kappa_+$-torsion (Lemma \ref{lemma.torsion}), it follows that  $\Gamma_1(\widehat{M})/{\rm Im}(\phi)$ is $\kappa_+$-torsion also. Besides, if $S_i$ is the last letter of $W_i$, then $Q_{W_i}(M)$ is $\kappa_{S_i}$-torsion-free, and so it is $\kappa_+$-torsion-free. In this way, $\Gamma_1(\widehat{M})$ is the limit of objects that are $\kappa_+$-torsion-free, and it is clear that $\Gamma_1(\widehat{M})$ is $\kappa_+$-torsion-free also. Since we have the short exact sequence 
    \[
    0\rightarrow Q_{\kappa_+}(M) \rightarrow \Gamma_1(\widehat{M}) \rightarrow \Gamma_1(\widehat{M})/{\rm Im}(\phi)\rightarrow 0,
    \]
    it follows that $Q_{\kappa_+}(M) \cong \Gamma_1(\widehat{M})$ \cite[Proposition 3.4]{Goldman}.

Finally, by recalling that $Q_W$ is an exact functor that commutes with finite limits, if $W\neq 1$ we have
    \[
    \Gamma_W(\widehat{M})\cong Q_W(\Gamma_1(\widehat{M}))\cong Q_W(Q_{\kappa_+}(M))\cong Q_W(M)=\widehat{M}(W).
    \]
    Therefore, by Remark \ref{sheaf.iff.limit} we conclude that $\widehat{M}$ is a sheaf.
    \end{proof}

Next, we define the notion of affine cover and quasi-coherent sheaf.

\begin{definition}
    An {\em affine} cover is a finite subset $\{T_i\mid i\in I\}$ of $\mathcal{O}$ such that $\bigcap_{i\in I}\mathscr{L}(T_i)=\mathscr{L}(\kappa_+)$.
\end{definition}

\begin{definition}
    A sheaf $\mathscr{F}$ is {\em quasi-coherent} if there exists an affine cover $\{T_i\mid i\in I\}$, and for each $i\in I$ there exists an $\mathsf{SG}$ $T_i^{-1}R$-module $M_i$ such that for all morphisms $V\rightarrow W$ in the category $\underline{\mathscr{W}}$, we have a commutative diagram given by
\begin{equation}\label{diagram.quasi}
\xymatrix{
   \mathscr{F}(T_iW) \ar[r]^{\rho_{T_iV}^{T_iW}}\ar[d] &  \mathscr{F}(T_iV)\ar[d] \\
   Q_W(M_i) \ar[r] & Q_V(M_i)
      }
\end{equation}
    where the vertical maps are isomorphisms in $\mathsf{LSG}-R$ and $Q_1(*):= Q_{\kappa_+}(*)$. $\mathscr{F}$ is called {\em coherent} if moreover all $M_i$ are finitely generated $\mathsf{SG}$ $T_i^{-1}R$-modules.
\end{definition}

\begin{remark}
    Note that the sheaf $\widehat{M}$ is quasi-coherent for each object in the category $\mathsf{LSG}-R$. If $M$ is finitely generated $\mathsf{SG}$ module, then $\widehat{M}$ is coherent.
\end{remark}

The proof of the following proposition is analogous to the proof of \cite[Theorem 1]{VanOystaeyenWillaert1995}. For the completeness of the paper, we include it here.

\begin{proposition}\label{iso.sheaves}
     If $\mathscr{F}$ is a quasi-coherent sheaf on $\mathscr{W}$ and $\Gamma_*(\mathscr{F})$ denotes its global sections $\mathscr{F}(1)$, then $\mathscr{F}$ is isomorphic to $\widehat{\Gamma_*(\mathscr{F})}$, the sheaf associated to $\Gamma_*(\mathscr{F})$.
\end{proposition}
\begin{proof}
First of all, notice that we can suppose that $M_i$ is $\kappa_+$-closed because if this is not the case, then we can replace it by $Q_{\kappa_+}(M_i)$ and the commutative diagram \ref{diagram.quasi} holds. We want to see that $\mathscr{F}(W)\cong Q_W(\Gamma_*(\mathscr{F}))$ for all $W\in\mathscr{W}$. If $W\neq 1$, then by Remark \ref{sheaf.iff.limit} and the fact that $Q_W$ commutes with finite limits (recall that $Q_W$ is an exact functor), it follows that $\mathscr{F}(W)$ is the limit of the diagram (\ref{limitdiagram}), while that $Q_W(\Gamma_*(\mathscr{F}))$ is the limit of the diagram
\begin{equation}\label{limitdiagram2}
\xymatrix@R=8pt{
    Q_W(\mathscr{F}(W_i)) \ar[dr] \ar[r] & Q_W(\mathscr{F}(W_iW_j))\\
    Q_W(\mathscr{F}(W_j)) \ar[ur] \ar[r] & Q_W(\mathscr{F}(W_jW_i))}
\end{equation}

Notic that we have the isomorphism $\mathscr{F}(T_i)\cong Q_{\kappa_+}(M_i)=M_i$, and by the diagram \ref{diagram.quasi}, for every $W$ there exists an isomorphism $\psi_i^W:Q_W(\mathscr{F}(T_i))\rightarrow \mathscr{F}(T_iW)$. If $W = S_1\dots S_n$ and $W_t := S_1\dots S_t$, then we obtain the following commutative diagram 
\begin{equation}\label{diagram.largo}
\begin{tiny}
\xymatrix{
   \mathscr{F}(T_i) \ar[r]^{\rho_{T_iW_1}^{T_i}} \ar[d] & \mathscr{F}(T_iW_1) \ar[r]^{\rho_{T_iW_2}^{T_iW_1}} \ar[d]^{\psi_i^{W_1}} & \mathscr{F}(T_iW_2) \ar[r] \ar[d]^{\psi_i^{W_2}} &\dotsb \ar[r] &\mathscr{F}(T_iW_{n-1})\ar[r]^{\rho_{T_iW}^{T_iW_{n-1}}} \ar[d]^{\psi_i^{W_{n-1}}} & \mathscr{F}(T_iW) \ar[d]^{\psi_i^{W}}\\
   \mathscr{F}(T_i) \ar[r] &  Q_{W_1}(\mathscr{F}(T_i))\ar[r] & Q_{W_2}(\mathscr{F}(T_i))\ar[r] &\dotsb \ar[r]& Q_{W_{n-1}}(\mathscr{F}(T_i)) \ar[r] &Q_{W}(\mathscr{F}(T_i))
      }
\end{tiny}
\end{equation}

Since $Q_{W_t}(\mathscr{F}(T_i))$ is an $S_t^{-1}R$-module, and so $\mathscr{F}(T_iW_t)$ also is, for an element $s_t\in S_t$ we can multiply by $s_t^{-1}$, whence the commutativity of the diagram above guarantees that 
\begin{small} 
    \begin{equation}\label{psi_i}
    \psi_i^W\left(\frac{1}{s_n}\otimes\dots\otimes\frac{1}{s_1}\otimes m\right)=s_n^{-1}\rho_{T_iW}^{T_iW_{n-1}}(s_{n-1}^{-1}\rho_{T_iW_{n-1}}^{T_iW_{n-2}}(\dots s_2^{-1}\rho_{T_iW_2}^{T_iW_1}(s_1^{-1}\rho_{T_iW_1}^{T_i}(m))\dots )).
    \end{equation}
\end{small}
    
On the other hand, we have 
\[
\mathscr{F}(T_iT_jW)\cong Q_{T_jW}(\mathscr{F}(T_i))=Q_W(Q_{T_j}(\mathscr{F}(T_i)))\cong Q_W(\mathscr{F}(T_iT_j)).
\]
If we write as $\psi_{ij}^W:Q_W(\mathscr{F}(T_iT_j))\rightarrow \mathscr{F}(T_iT_jW)$ the isomorphism above, by using a similar diagram to the above it can be seen that  

\begin{tiny} 
    \begin{equation}\label{psi_ij}
     \psi_{ij}^W\left(\frac{1}{s_n}\otimes\dots\otimes\frac{1}{s_1}\otimes m\right) = s_n^{-1}\rho_{T_iT_jW}^{T_iT_jW_{n-1}}(s_{n-1}^{-1}\rho_{T_iT_jW_{n-1}}^{T_iT_jW_{n-2}}(\dots s_2^{-1}\rho_{T_iT_jW_2}^{T_iT_jW_1}(s_1^{-1}\rho_{T_iT_jW_1}^{T_i}(m))\dots )).
    \end{equation}
\end{tiny}

Notice that $\rho_{T_iT_jW_t}^{T_iW_t}$ and $\rho_{T_jT_iW_t}^{T_iW_t}$ are $S_t^{-1}R$-linear for $t = 1,\dotsc,n$, since both are $R$-homomorphisms between $S_t^{-1}R$-modules. In this way, the expressions (\ref{psi_i}) and (\ref{psi_ij}) imply the commutativity of the following two diagrams 
\begin{equation}\label{diagram.doble}
\begin{tabular}{cc}
\begin{minipage}{2.2in}
\begin{displaymath}
\leftline{\xymatrixcolsep{3pc}\xymatrix{
    Q_W(\mathscr{F}(T_i))\ar[r]^{Q_W\left(\rho_{T_iT_j}^{T_i}\right)} \ar[d]^{\psi_i} & Q_W(\mathscr{F}(T_iT_j))\ar[d]^{\psi_{ij}} \\
    \mathscr{F}(T_iW)) \ar[r]^{\rho_{T_iT_jW}^{T_iW}} & \mathscr{F}(T_iT_jW) }}
\end{displaymath}
\end{minipage}
& 
\begin{minipage}{2.2in}
\begin{displaymath}
\leftline{\xymatrixcolsep{3pc}\xymatrix{
     Q_W(\mathscr{F}(T_i)) \ar[r]^{Q_W\left(\rho_{T_jT_i}^{T_i}\right)} \ar[d]^{\psi_i}  & Q_W(\mathscr{F}(T_jT_i)) \ar[d]^{\psi_{ji}}\\
     \mathscr{F}(T_iW) \ar[r]^{\rho_{T_jT_iW}^{T_iW}} & \mathscr{F}(T_jT_iW)
      }}
\end{displaymath}
\end{minipage}
\end{tabular}
\end{equation}

Hence, it is clear that diagrams (\ref{sheaf.iff.limit}) and (\ref{limitdiagram2}) must to have isomorphic limits, that is, $\mathscr{F}\cong Q_W(\mathscr{F}(\Gamma_*))$. Besides, for a morphism $V\rightarrow W$ the map $\rho_V^W$ is determined by the maps $\rho_{T_iV}^{T_iW}$ and $\rho_{T_iT_jV}^{T_iT_jW}$, which shows that the diagram 
\begin{equation*}
    \xymatrix{
    Q_W(\Gamma_*(\mathscr{F})) \ar[r]\ar[d] & Q_v(\Gamma_*(\mathscr{F}))\ar[d]\\
    \mathscr{F}(W)\ar[r]^{\rho_v^w} & \mathscr{F}(V)
    }
\end{equation*}
is commutative. 

For $W = 1$ we have to show that $\Gamma_*(\mathscr{F})\cong Q_{\kappa_+}(\Gamma_*(\mathscr{F}))$. Since $\mathscr{F}(T_i)$ and $\mathscr{F}(T_iT_j)$ are $\kappa_+$-torsion-free ($\mathscr{F}(T_i)\cong M_i$ and  $\mathscr{F}(T_iT_j)\cong T_j^{-1}(\mathscr{F}(T_i))$), then $\Gamma_*(\mathscr{F})$ is $\kappa_+$-torsion-free because it is the limit of objects $\kappa_+$-torsion-free. Let us see that $\Gamma_*(\mathscr{F})$ is $\kappa_+$-injective. By \cite[Proposition 3.2]{Goldman}, it is sufficient to show that for all $I\in\mathscr{L}(\kappa_+)$, every $R$-homomorphism $f: I\rightarrow \Gamma_*(\mathscr{F})$ can be extended to a $R$-homomorphism $g: R\rightarrow \Gamma_*(\mathscr{F})$.

Since $\mathscr{F}(T_i)$ is $\kappa_+$-injective, the map $\rho_{T_i}\circ f$ can be extended to a map $g_i:R\rightarrow \mathscr{F}(T_i)$. If $x_i = g_i(1)$, then $g_i(r)=rx_i$ for every $r\in R$. In particular, for each $a\in I$ we have that  $a\rho_{T_iT_j}^{T_i}(x_i)=\rho_{T_iT_j}^{T_i}(\rho_{T_i}(f(a)))=\rho_{T_iT_j}^{T_j}(\rho_{T_j}(f(a)))=\rho_{T_iT_j}^{T_j}(x_j)$, which shows that there exists an element $x\in \Gamma_*(\mathscr{F})$ such that $\rho_{T_i}(x)=x_i$, for every $i$. Notice that the map $g:I\rightarrow \Gamma_*(\mathscr{F})$ defined by $g(r) = rx$ extends $f$, so we conclude $\Gamma_*(\mathscr{F})= Q_{\kappa_+}(\Gamma_*(\mathscr{F}))$.
\end{proof}

\begin{theorem}\label{th.quasi-coherent}
The category of quasi-coherent sheaves is equivalent to the category $(R,\kappa_+)-\mathsf{LSG}$.
\end{theorem}
\begin{proof}
Let $\mathscr{F}$ be a quasi-coherent sheaf. From the last part of the proof of the Proposition \ref{iso.sheaves}, we have that $\Gamma_*(\mathscr{F})$ is an object of $(R,\kappa_+)-\mathsf{LSG}$. Moreover, if $M$ belongs to $(R,\kappa_+)-\mathsf{LSG}$, then $M\cong Q_{\kappa_+}(M)$. In this way, $\widehat{\square}$ and $\Gamma_*\square$ are functors between the category $(R,\kappa_+)-\mathsf{LSG}$ and the category of quasi-coherent sheaves, which are equivalences by Propositions \ref{prop.tildeM.sheaf} and \ref{iso.sheaves}. 
\end{proof}

Finally, we arrive to the most important result of the paper: the {\em Serre-Artin-Zhang-Verevkin} theorem for semi-graded rings.

\begin{theorem}{\bf (Serre-Artin-Zhang-Verevkin theorem).}\label{Serre-theo}
The category of coherent sheaves is equivalent to ${\rm Proj}(R)$.
\begin{proof}
From Theorem \ref{th.quasi-coherent}, it is sufficient to show that $M$ belongs to ${\rm Proj}(R)$ if and only if $\widehat{M}$ is coherent.
    
We fix a cover $\{T_i\mid i\in I\}$. If $M$ is an element of ${\rm Proj}(R)$, then there exist $m_1, \dotsc ,m_k\in M$ such that $M/N$ is $\kappa_+$-torsion with $N = \langle m_1,\dotsc,m_k\rangle^{\mathsf{SG}}$. Let $f_i: M\rightarrow T_i^{-1}M$ be the canonical map. It is straightforward to see that $\langle \frac{m_1}{1},\dots,\frac{m_k}{1}\rangle^{\mathsf{SG}}_R=f_i(N)$, and that the $T_i^{-1}M$-submodule $J_i=\langle \frac{m_1}{1},\dots,\frac{m_k}{1}\rangle^{\mathsf{SG}}_{T_i^{-1}R}$ satisfies the relation $f_i(N)\subseteq J_i$. Let $m\in M$. Since $M/N$ is $\kappa_+$-torsion, there exists $I\in\mathscr{L}(\kappa_+)$ such that $Im\subseteq N$. By using that $T_i$ is non-trivial, there exists $t_i\in I\cap T_i$, whence $t_im\in N$, which shows that $\frac{t_im}{1}\in f_i(N)$. Since $J_i$ is a $T_i^{-1}R$-module and $\frac{t_im}{1}\in J_i$, then $J_i = T_i^{-1}M$, that is, $T_i^{-1}M$ is finitely generated as an $\mathsf{SG}$ $T_i^{-1}R$-module.

Suppose that every one of the $t_i^{-1}M$ is a finitely generated $\mathsf{SG}$ $T_i^{-1}R$-module. Note that if $T_i^{-1}M = \langle \frac{m_{1,i}}{s_{1,i}},\dots,\frac{m_{t_i,i}}{s_{t_i,i}}\rangle^{\mathsf{SG}}_{T_i^{-1}R}$, then $T_i^{-1}M=\langle \frac{m_{1,i}}{1},\dots,\frac{m_{t_i,i}}{1}\rangle^{\mathsf{SG}}_{T_i^{-1}R}$. Since there are finite elements $i$'s, then the union set $\bigcup_{i\in I} \{m_{1,i},\dots,m_{t_i,i}\}$ is also finite, $\{m_1,\dots,m_k\}$ say. If we define $N = \langle m_1,\dots,m_k\rangle^{\mathsf{SG}}$, it is straightforward to see that $T_i^{-1}N$ is an $\mathsf{SG}$ $T_i^{-1}R$-submodule of $T_i^{-1}M$, whence $T_i^{-1}N = T_i^{-1}M$. 

For an element $m\in M$, since $\frac{m}{1}\in T_i^{-1}N$ there exist elements $n_i\in N$ and $t_i\in T_i$ such that $\frac{m}{1}=\frac{n_i}{t_i}$. Hence, there exist $c_i,d_i\in R$ such that $c_im = d_in_i$ and $c_i = d_it_i\in T_i$, whence $c_im\in N$. By using that $c_i\in T_i$ for each $i\in I$ and that $\{T_i\mid \in I\}$ is an affine cover, it follows that $I = \sum Rc_i\in\mathscr{L}(\kappa_+)$. We conclude that $Im\subseteq N$, and therefore $M/N$ is $\kappa_+$-torsion.
\end{proof}
\end{theorem}
%Next, we illustrate our results with some remarkable algebras appearing in the literature.

Next, we show that the notion of schematicness in the semi-graded setting generalizes the corresponding concept in the case of connected and $\mathbb{N}$-graded algebras introduced and studied by Van Oystaeyen and Willaert \cite{VanOystaeyen2000, VanOystaeyenWillaert1995, vanOystaeyenWillaert1996, VanOystaeyenWillaert1997, Willaert1998}.

\begin{remark}\label{we-generalize}
Consider a positively graded left Noetherian ring $R$. It is clear that $R_+ = R_{\ge 1}$. Note that if $R$ is generated in degree one, then $R_{\ge t} = (R_+)^t$, which shows that $\mathscr{L}(\kappa_{+}) = \{I\vartriangleleft_l R\mid {\rm there\ exists}\ n\in \mathbb{N}\ {\rm with}\ (R_{+})^n\subseteq I\}$. On the other hand, the $\mathsf{LSG}$ modules are exactly the same $\mathbb{N}$-graded modules, and the good left Ore sets coincide with the homogeneous left Ore sets. In this way, the notion of schematic ring presented in this paper generalizes the corresponding notion introduced by Van Oystaeyen and Willaert \cite{VanOystaeyenWillaert1995}. Last, but not least, notice that in the $\mathbb{N}$-graded setting, the left noetherianity of $R$ implies that the finitely generated objects of $(R,\kappa_+)-\mathsf{LSG}$ are Noetherian, which shows that Theorem \ref{Serre-theo} generalizes \cite[Theorem 3]{VanOystaeyenWillaert1995}.
\end{remark}

Next, we show some examples that illustrate our Theorem \ref{Serre-theo} in the case of non-$\mathbb{N}$-graded rings where \cite[Theorem 3]{VanOystaeyenWillaert1995} cannot be applied.

\begin{example}\label{examplesillustrating}
\begin{enumerate}
    \item [(i)] Consider the first Weyl algebra $A_1(\Bbbk) = \Bbbk\{ x,y\}/ \langle yx-xy-1\rangle$ over a field $\Bbbk$ of ${\rm char}(\Bbbk) = p>0$. It is well-known that $A_1(\Bbbk)$ is a non-$\mathbb{N}$-graded ring, the set $\{x^ny^m\mid n, m\in\mathbb{N}\}$ is $\Bbbk$-basis of $A_1(\Bbbk)$, and $A_1(\Bbbk)$ is a Noetherian ring. Since $x^p, y^p\in Z(A_1(\Bbbk))$, it is clear that $\{x^{pk}\mid k\in\mathbb{N}\}$ and $\{y^{pk}\mid k\in\mathbb{N}\}$ are good left Ore sets. Besides, if $k_1,k_2\in\mathbb{N}$ then $A_1(\Bbbk)x^{pk_1} + A_1(\Bbbk)y^{pk_2}$ is a two-sided ideal of $A_1(\Bbbk)$ which is a left $\mathsf{SG}$ submodule, whence $A_1(\Bbbk)_{\ge pk_1+pk_2}\subseteq A_1(\Bbbk)x^{pk_1} + A_1(\Bbbk)y^{pk_2}$. Therefore, $A_1(\Bbbk)$ is a schematic algebra and Theorem \ref{Serre-theo} holds.
    \item [(ii)] In a similar way, it can be shown that the $n$-th Weyl algebra $A_n(\Bbbk)$ is schematicness when ${\rm char}(\Bbbk) = p>0$.
    \item [(iii)] The well-known Jordan plane $\Bbbk\{x,y\}/\langle yx-xy-y^2\rangle$ is schematic when ${\rm char}(\Bbbk) = p>0$ since the sets $\{x^{pk}\mid k\in\mathbb{N}\}$ and $\{y^{pk}\mid k\in\mathbb{N}\}$ are good left Ore sets.
\end{enumerate}
\end{example}

For the skew PBW extensions introduced by Gallego and Lezama \cite{GallegoLezama2011} (c.f. \cite{ReyesRodriguez2021, ReyesSuarez2021}), which are examples of non-$\mathbb{N}$-graded rings, Proposition \ref{prop.center-esquematic} establishes sufficient conditions to guarantee their schematicness.

\begin{proposition}\label{prop.center-esquematic} 
Let $A = \sigma(R)\langle x_1,\dotsc, x_n\rangle$ be a bijective skew PBW extension over a left Noetherian ring $R$ with the usual semi-graduation, that is, ${\rm deg}(x_i) = 1$ and ${\rm deg}(r) = 0$, for every $i$ and each $r\in R$. If for every $i$ there exists $m_i\ge 1$ such that $x_i^{m_i}\in Z(A)$, then $A$ is schematic.
\end{proposition}
\begin{proof}
From \cite[Theorem 3.1.5]{libropbw} we know that $A$ is left Noetherian. Since $x_i^{m_i}\in Z(A)$, it follows that $\{x_i^{m_im}\mid m\in\mathbb{N}\}$ is a non-trivial good left Ore Set for every $i$. Let us see that these sets satisfy the schematicness condition. Let $t_i\in \mathbb{N}$. Then $\sum_{i=1}^n Rx_i^{m_it_i}$ is a two-sided ideal and an $\mathsf{SG}$ submodule of $A$. Besides, if $t := \sum m_it_i$ then $\bigoplus_{m\ge t}R_m\subseteq \sum_{i=1}^n Rx_i^{m_it_i}$, whence $R_{\ge t}\subseteq \sum_{i = 1}^{n} Rx_i^{m_it_i}$.
\end{proof}

Examples \ref{firstindep} and \ref{secondindep} show that the theory presented by Lezama about Serre-Artin-Zhang-Verevkin theorem and the one developed in this paper are independent.

\begin{example}\label{firstindep}
Proposition \ref{prop.center-esquematic} guarantees that if $R$ is a left Noetherian noncommutative ring, then $ A = R[x]$ is schematic, and so Theorem \ref{Serre-theo} holds for $A$. Notice that this result cannot be obtained from the the theory developed by Lezama \cite{Lezama2021, LezamaLatorre2017} because does not satisfy Lezama's assumption (C4) that says that $A_0 = R$ is a commutative ring. Notice that in the particular case of the $\Bbbk$-algebra $R = M_n(\Bbbk)$, since $R$ is not connected it does not satisfy the definition of schematicness given by Van Oystaeyen and Willaert (Definition \ref{schematicgradedsetting}), and it is not a {\em finitely semi-graded algebra} in the sense of Lezama \cite[Definition 2.4]{LezamaLatorre2017}. However, from our point of view, the algebra is schematic and Theorem \ref{Serre-theo} holds.
\end{example}

\begin{example}\label{secondindep}
Consider $A$ as the 3-dimensional skew polynomial algebra subject to the relations
\begin{equation*}
    yz=zy, \quad xz=zx, \quad {\rm and}\quad yx=xy-z. 
\end{equation*}
This algebra satisfies the Serre-Artin-Zhang-Verevkin theorem \cite[Example 18.5.15 (v)]{libropbw} following the ideas presented by Lezama \cite{Lezama2021, LezamaLatorre2017}

It is straightforward to see that the following relations hold:
\begin{equation*}
    y^nx=xy^n-ny^{n-1}z,\quad {\rm and}\quad
    yx^n=x^ny-nx^{n-1}z,\quad n > 0.
    \end{equation*}
\end{example}
If ${\rm char}(\Bbbk)=p>0$, then $x^p,y^p,z\in Z(A)$, and so Proposition \ref{prop.center-esquematic} implies that $A$ is schematic. 

\medskip

Next, consider the case ${\rm char}(\Bbbk) = 0$. Let us see that $A''_n=\{az^n\mid a\in\Bbbk, n\in\mathbb{N}\}$. With this aim, consider $\alpha\in A''_n$. Then $\alpha$ is a homogeneous element of degree $n$, and so
\[
\alpha=\sum_{i+j\leq n}a_{i,j}x^iy^jz^{n-i-j} .
\]
Since that
   \begin{align*}
                \alpha x&=\sum_{i+j\leq n}a_{i,j}x^iy^jxz^{n-i-j}\\
                &= \sum_{i+j\leq n}a_{i,j}x^i(xy^j-jy^{j-1z})z^{n-i-j}\\
                &= \sum_{i+j\leq n}a_{i,j}x^{i+1}y^jz^{n-i-j}-\sum_{i+j\leq n}ja_{i,j}x^iy^{j-1}z^{n-i-j+1},
\end{align*}
the element $\alpha x$ is homogeneous of degree $n+1$, whence $ja_{i,j}=0$, for each $i, j$. In a similar way, for the element $y\alpha$ we obtain $ia_{i,j}=0$, for each $i,j$. These facts imply that the only non-zero coefficient is precisely $a_{0,0}$, and so $\alpha = a_{0,0}z^n$. This shows that $A''_n=\{az^n\mid a\in\Bbbk, n\in\mathbb{N}\}$.

Now, let us prove that $A$ is not schematic. Since $z\in Z(A)$, then $S=\{az^k\mid a\in\Bbbk^*, k\in\mathbb{N}\}$ is a good left Ore set. Note that for all $m\in \mathbb{N}$, $x^m\in A_{\ge m}\backslash Rz$, whence $S$ does not satisfy the schematicness condition. Besides, due to the reasoning above, it is clear that $S$ contains any other good left Ore of $A$, and so if $S$ does not satisfy the schematicness condition, then no other set will.

\medskip

Finally, Proposition \ref{propositionresume} presents necessary conditions to assert the schematicness of skew PBW extensions with two indeterminates.

\begin{proposition}\label{propositionresume}
Let $A = \sigma(\Bbbk)\langle x,y\rangle$ be a skew PBW extension over $\Bbbk$ defined by the relation
\begin{equation}\label{SPBWtwoindeter}
    yx = dxy + ex + fy + g,\quad d, e, f, g\in \Bbbk.
\end{equation}
$A$ is schematic if and only if one of the following cases holds:
\begin{enumerate}
\item [\rm (1)] $yx = dxy$ {\rm (}quantum plane{\rm )}.
\item [\rm (2)] $yx=xy + g$ with ${\rm char}(\Bbbk)=p>0$.
\item [\rm (3)] $yx = dxy + g$ with $d \neq 1$ and $d^p = 1$ for some $p\in \mathbb{N}$.
\end{enumerate}
\end{proposition}
\begin{proof}
We divide the proof into four parts.
\begin{enumerate}
\item [\rm (a)] Let $P := d x+ f$, $Q := ex + g$, $\overline{P}: = d y + e$ and $\overline{Q} :=  fy + g$. Notice that the binomial theorem holds for $P$ and $\overline{P}$, that is,
 \begin{equation*}
     P^i=\sum_{k=0}^i \binom{i}{k} d^{i-k}f^k x^{i-k}, \quad {\rm and}\quad 
     \overline{P}^i=\sum_{k=0}^i \binom{i}{k} d^{i-k} e^k y^{i-k},\ {\rm for\ all}\ i>0,
 \end{equation*}

and that $yx=Py+Q=x\overline{P} +\overline{Q}$. 

\medskip

Let us see some relations of commutativity between $x$ and $y$.

\medskip

For $n\ge 1$, the following identities  
\begin{align}
             yx^n &= P^ny+\sum_{i=0}^{n-1} P^{n-1-i}x^iQ,\ {\rm and}\\
             y^nx &= x\overline{P}^{n}+\sum_{i=0}^{n-1} \overline{P}^{n-1-i}y^i\overline{Q}
\end{align}
hold.

\medskip

The case $n = 1$ is clear. Suppose that the assertion holds for $n$. Then
        \begin{align*}
            yx^{n+1} &=  \left(P^ny+\sum_{i=0}^{n-1} P^{n-1-i}x^iQ\right)x\\
            &=P^n(Py+Q)+\sum_{i=0}^{n-1} P^{n-1-i}x^{i+1}Q\\
            &= P^{n+1}y+(P^n+\sum_{i=0}^{n-1} P^{n-1-i}x^{i+1})Q\\
            &=P^{n+1}y+\sum_{i=0}^{n} P^{n-i}x^{i}Q,
        \end{align*}
which concludes the proof. In a similar way, we can prove the other equality.

\item [\rm (b)] For $n>0$, we write $\Delta_n :=\sum_{i=0}^{n-1}d^i$.

\medskip

Let us see that if $\xi = ax^n$ {\rm (}resp. $\xi = ay^n${\rm )} belongs to $R''_n$ with $a\neq 0$, then $f = 0$ {\rm (}resp. $e = 0${\rm )}. In the case $Q\neq 0$ {\rm (}resp. $\overline{Q}\neq 0${\rm )}, it follows that $\Delta_n = 0$.

\medskip

The equalities 
        \begin{align*}
            y\xi & =ayx^n\\
            & = a\left(P^ny+\sum_{i=0}^{n-1}P^{n-1-i}x^iQ \right)\\
            & = a\left(\sum_{k=0}^n\binom{n}{k}d^{n-k}f^kx^{n-k}y+\sum_{i=0}^{n-1}P^{n-1-i}x^iQ  \right),
\end{align*}
show that the element $y\xi$ is homogeneous of degree $n+1$, and that the monomials having $y$ satisfy that if $k\neq 0$ then $a\binom{n}{k}d^{n-k}f^{k}=0$. In particular, if $k = n$, then $af^{n} = 0$, whence $f = 0$.

\medskip

Now, with respect with the other monomials, it is clear that these form a polynomial element of degree less than $n+1$, which shows that  
\[ 
a\sum_{i=0}^{n-1}P^{n-1-i}x^iQ =0.
\]

Since $f = 0$, $P = dx$, and so
\[
0 = a\sum_{i=0}^{n-1}P^{n-1-i}x^iQ = a\left(\sum_{i=0}^{n-1}d^{n-1-i}\right)x^{n-1}Q.
\]
Thus, if $Q\neq0$ then 
\[
0 = \sum_{i=0}^{n-1}d^{n-1-i}=\sum_{i=0}^{n-1}d^{i}=\Delta_n.
\]

The proof for the case $ay^n$ is analogous.

\item [\rm (c)]
Let $n\ge 1$ and
\[
\xi=\sum_{i=0}^na_ix^iy^{n-i}\in R''_n.
\]

Let us show that if $\overline{Q}\neq 0$ {\rm (}resp. $Q\neq 0${\rm )} and $\xi\neq a_nx^n$ {\rm (}resp. $\xi\neq a_0y^n${\rm )}, then $e = 0$ {\rm (}resp. $f = 0${\rm )} and  $\Delta_k = 0$ for some $0\leq k\leq n$.

\medskip

Note that $\xi x$ is a homogeneous element of degree $n+1$. We have the equalities given by
 \begin{equation}\label{equation.xi.x}
    \begin{split}
    \xi x&=a_nx^{n+1}+\sum_{i=0}^{n-1}a_ix^iy^{n-i}x\\
     &= a_nx^{n+1}+\sum_{i=0}^{n-1}a_ix^i\left(x\overline{P}^{n-i}+\sum_{j=0}^{n-1-i}\overline{P}^{n-1-i-j}y^j\overline{Q} \right)\\
     &=a_nx^{n+1}+\sum_{i=0}^{n-1}\left(a_ix^{i+1}\overline{P}^{n-i}+a_ix^i\sum_{j=0}^{n-1-i}\overline{P}^{n-1-i-j}y^j\overline{Q} \right).
     \end{split}
 \end{equation}

Suppose that there exists $0\le i\le n-1$ such that $a_i\neq 0$, and let $t := {\rm min}\{0\le i\le n\mid a_i\neq 0\}$. Then
\[
\xi x=a_nx^{n+1}+\sum_{i=t}^{n-1}\left(a_ix^{i+1}\overline{P}^{n-i}+a_ix^i\sum_{j=0}^{n-1-i}\overline{P}^{n-1-i-j}y^j\overline{Q}\right).
\]
Notice that the lower exponent of $x$ appears when $i = t$, and we have that 
\[
a_tx^t\sum_{j=0}^{n-1-t}\overline{P}^{n-1-t-j}y^j\overline{Q}
\]
is a polynomial element of degree less than $n+1$ that have no another terms of $\xi x$, whence necessarily this polynomial has to be the zero element. Since $a_t\neq 0$, it follows that
\[
\left(\sum_{j=0}^{n-1-t}\overline{P}^{n-1-t-j}y^j\right)\overline{Q} = 0. 
\]
By using that $\overline{Q} \neq 0$, we have $\sum_{j=0}^{n-1-t}\overline{P}^{n-1-t-j}y^j=0$. Hence, 
\begin{equation}\label{eq.oQ.no.0}
\begin{split}
    0 &=\sum_{j=0}^{n-1-t}\overline{P}^{n-1-t-j}y^j\\
    &= \sum_{j=0}^{n-1-t}\left( \sum_{k=0}^{n-1-t-j} \binom{n-1-t-j}{k} d^{n-1-t-j-k}e^k y^{n-1-t-j-k}\right)y^j\\
    &= \sum_{j=0}^{n-1-t}\sum_{k=0}^{n-1-t-j} \binom{n-1-t-j}{k} d^{n-1-t-j-k}e^k y^{n-1-t-k}.
 \end{split}
\end{equation}

The coefficient of the monomial $y^0$ is obtained when $j=0$ and $k=n-1-t$, which implies that 
\[
\binom{n-1-t}{n-1-t}d^{n-1-t-0-(n-1-t)}e^{n-t-1} = e^{n-t-1} = 0,
\]
whence $n-1\neq t$ and $e = 0$, and by replacing in the expression (\ref{eq.oQ.no.0}), it follows that 
\[
0 = \sum_{j=0}^{n-1-t}d^{n-1-t-j}y^{n-1-t},
\]
and so
\[
0 = \sum_{j=0}^{n-1-t}d^{n-1-t-j}=\sum_{j=0}^{n-1-t}d^{j}=\Delta_{n-t}.
\]

The condition that there exists $n > 0$ such that $\Delta_n=0$ is recursive, so will call it {\em Condition U}. It is straightforward to see that this condition is satisfied if and only if one of the following conditions hold:
\begin{itemize}
    \item $d = 1$ and ${\rm char}(\Bbbk)=p>0$.
    \item $d \neq 1$ and there exists $p > 0$ such that $d^p = 1$.
\end{itemize}

\item [\rm (d)]  With the analysis above, we can determine the schematicness of the skew PBW extensions defined by relation (\ref{SPBWtwoindeter}). 

\medskip

First of all, note that if $d = 0$, Part (c) implies that $A''=\Bbbk$ since the {\em Condition U} does not hold. Thus, the skew PBW extension $A$ is not schematic. From now on, consider $d\neq 0$. It is clear that the case $e = f = g = 0$ shows that $A$ is schematic. Let us see what happens if one of these three elements is non-zero and the {\em Condition U} does not hold.

\medskip

Let $\xi\in A''_n$ with $n > 1$. If $g \neq 0$ then $Q \neq 0 \neq \overline{Q}$, and by Part (c), we have $\xi = a_nx^n=a_0y^n$, whence $\xi=0$, and so $A'' = \Bbbk$, which shows that $A$ is not schematic. If $e \neq 0$, then $Q\neq 0$, and Part (c) implies that $\xi = a_0y^n$, whence $\xi=0$ (Part (b)). In this case, $A$ is not schematic. Similarly, if $f \neq 0$ then $A$ is not schematic.

\medskip

Let us see the case where {\em Condition U} holds (with the less value of $p$ satisfying this condition), and two of the three elements $e, f, g$ being non-zero. If $e\neq 0\neq f$, then Part (c) implies that $\xi = 0$, whence $A$ is not schematic. Now, if $e \neq 0\neq g$ and $f = 0$, it follows that $x^p\in Z(R)$. On the other hand, $Q \neq 0 \neq \overline{Q}$ and so Part (c) shows that $\xi = a_nx^n$ with $p\mid n$. In this way, $A''=\{ax^{pk}\mid a\in\Bbbk, k\in\mathbb{N}\}$, and hence $S = \{ax^{pk}\mid a\in\Bbbk^*, k\in\mathbb{N} \}$ is the greatest left Ore set, and since $S$ does not satisfy the condition of schematicness (due to the powers of $y$), it is clear that $A$ is not schematic. Analogously, one can check that if $f \neq 0 \neq g$ and $e = 0$, then $A$ is not schematic.

\medskip

Now, let us see the situation where {\em Condition U} is satisfied and only one element is non-zero. If $e \neq 0$ and $f = 0 = g$, then $\overline{Q} = 0$, and so equation \ref{equation.xi.x} can be written as
\[
\xi x = a_nx^{n+1}+\sum_{i=0}^{n-1}\sum_{k=0}^{n-i}\binom{n-i}{k}a_i d^{n-i-k} e^kx^{i+1}y^{n-i-k}.
\]
Note that every value of $i$ corresponds to only one power of $x$, and when $k\neq 0$ the degree of $x^{i+1}y^{n-i-k}$ is less than $n+1$. These facts show that $ \binom{n-i}{k}a_id^{n-i-k} e^k=0$, for each $0\le i\le n-1$ and all $0<k\le n-i$. In particular, if $k=1$ then $a_i=0$, and so $\xi = a_nx^n$. Part (b) implies that $p\mid n$, whence $A$ is not schematic by the same reason as above in the case $e \neq 0 \neq g$ and $f = 0$. Analogously, if $f \neq 0$ and $e = 0 = g$, it follows that $A$ is not schematic.

\medskip

Finally, if the {\em Condition U} holds and $g \neq 0$ with $e = 0 = f$, then it is straightforward to see that $x^p, y^p\in Z(R)$, whence $A$ is schematic. 
\end{enumerate}
\end{proof}

\begin{remark}
Proposition \ref{propositionresume} shows that there are Ore extensions over schematic rings that are not schematic. This is consistent with Proposition \ref{VanOystaeyenWillaert1997Theorem3}.
\end{remark}

\section{Conclusions and future work}\label{conclusionsfuturework}

In this paper, we have defined the notion of schematic ring in the context of semi-graded objects and illustrated our Theorem \ref{Serre-theo} with some non-$\mathbb{N}$-graded algebras. With the aim of obtaining new examples of schematic algebras in this more general setting, it is of interest to generalize the criterion formulated by Van Oystayen and Willaert \cite[Lemma 2]{VanOystaeyenWillaert1997} that says that if $R$ is an $\mathbb{N}$-graded $\Bbbk$-algebra such that its center $Z(R)$ is Noetherian and such that $R$ is a finitely generated $Z(R)$-module, then $R$ is schematic. The importance of this criterion can be appreciated in \cite{ChaconReyes2022} where the authors investigated the schematicness of {\em skew Ore polynomials of higher order generated by homogenous quadratic relations} defined by Golovashkin and Maksimov \cite{GolovashkinMaksimov1998, GolovashkinMaksimov2005}. Since these algebras are non-$\mathbb{N}$-graded, the research on its schematicness will be crucial for another families of noncommutative rings.

\medskip

Now, having in mind that Willaert \cite{Willaert1998} studied the least possible number of Ore sets satisfying the condition of schematicness for $\mathbb{N}$-graded algebras, and call it the {\em schematic dimension}, a natural task is to investigate this notion in the setting of semi-graded rings. Also, an important topic of future research for these rings is the \v{C}ech cohomology developed by Van Oystayen and Willaert \cite{vanOystaeyenWillaert1996, VanOystaeyenWillaert1997}.

\end{document}